\documentclass[a4paper, twoside]{article}
\usepackage{a4}
\usepackage{amssymb}
\usepackage{amsmath}
\usepackage{aligned-overset}
\usepackage[active]{srcltx}
\usepackage[pagebackref,colorlinks,citecolor=blue,linkcolor=blue,urlcolor=blue]{hyperref}
\usepackage[dvipsnames]{color}
\usepackage{upref}
\allowdisplaybreaks[2] 
\newcount\minutes \newcount\hours
\hours=\time
\divide\hours 60
\minutes=\hours
\multiply\minutes -60
\advance\minutes \time
\newcommand{\klockan}{\the\hours:{\ifnum\minutes<10 0\fi}\the\minutes}
\newcommand{\tid}{\today\ \klockan}
\newcommand{\prtid}{\smash{\raise 10mm \hbox{\LaTeX ed \tid}}}
\renewcommand{\prtid}{}
\makeatletter
\def\sectionmark#1{} 
\def\subsectionmark#1{}
\newcommand{\sectnr}{\ifnum \c@secnumdepth >\z@
                 \thesection.\hskip 1em\relax \fi}
\def\@evenhead{\footnotesize\rm\thepage\hfil\leftmark\hfil\llap{\prtid}}
\def\@oddhead{\footnotesize\rm\rlap{\prtid}\hfil\rightmark\hfil\thepage}
\def\tableofcontents{\section*{Contents} 
 \@starttoc{toc}}
\makeatother
\makeatletter
\def\@biblabel#1{#1.}
\makeatother
\makeatletter
\let\Thebibliography=\thebibliography
\renewcommand{\thebibliography}[1]{\def\@mkboth##1##2{}\Thebibliography{#1}
\addcontentsline{toc}{section}{References}
\frenchspacing 
\setlength{\@topsep}{0pt}
\setlength{\itemsep}{0pt}%
\setlength{\parskip}{0pt plus 2pt}%
}
\makeatother
\makeatletter
\def\mdots@{\mathinner.\nonscript\!.%
 \ifx\next,.\else\ifx\next;.\else\ifx\next..\else
 \nonscript\!\mathinner.\fi\fi\fi}
\let\ldots\mdots@
\let\cdots\mdots@
\let\dotso\mdots@
\let\dotsb\mdots@
\let\dotsm\mdots@
\let\dotsc\mdots@
\def\vdots{\vbox{\baselineskip2.8\p@ \lineskiplimit\z@
    \kern6\p@\hbox{.}\hbox{.}\hbox{.}\kern3\p@}}
\def\ddots{\mathinner{\mkern1mu\raise8.6\p@\vbox{\kern7\p@\hbox{.}}%
    \raise5.8\p@\hbox{.}\raise3\p@\hbox{.}\mkern1mu}}
\def\cdots{\mathinner{\mkern1mu{\cdot}{\cdot}{\cdot}\mkern1mu}}
\makeatother
\makeatletter
\let\Enumerate=\enumerate
\renewcommand{\enumerate}{\Enumerate%
\setlength{\itemsep}{0pt}%
\setlength{\parskip}{0pt plus 1pt}%
\renewcommand{\theenumi}{\textup{(\alph{enumi})}}%
\renewcommand{\labelenumi}{\theenumi}%
}
\makeatother
\makeatletter
\def\@seccntformat#1{\csname the#1\endcsname.\quad}
\makeatother
\makeatletter
\long\def\@makecaption#1#2{%
  \vskip\abovecaptionskip
  \sbox\@tempboxa{ #1. #2}%
  \ifdim \wd\@tempboxa >\hsize
    #1. #2\par
  \else
    \global \@minipagefalse
    \hb@xt@\hsize{\hfil\box\@tempboxa\hfil}%
  \fi
  \vskip\belowcaptionskip}
\makeatother
\RequirePackage{ifthen}
\newcommand{\authortitle}[3]{\author{#1}\title{#2}%
   \ifthenelse{\equal{#3}{}}{\markboth{#1}{#2}}{\markboth{#1}{#3}}}
\newcommand{\art}[6]{{\sc #1, \rm #2, \it #3\/ \bf #4 \rm (#5), \mbox{#6}.}}
\newcommand{\auth}[2]{{#1, #2.}}
\def\idxauth{\auth}

\newcommand{\artprep}[3]{{\sc #1, \rm #2, #3.}}
\newcommand{\artin}[3]{{\sc #1, \rm #2, in #3.}}
\newcommand{\book}[3]{{\sc #1, \it #2, \rm #3.}}
\newcommand{\AND}{{\rm and }}
\newcommand{\arXiv}[1]{{\tt \href{https://arxiv.org/abs/#1}{arXiv:#1}}}
\RequirePackage{amsthm}
\newtheoremstyle{descriptive}%
  {\topsep}   
  {\topsep}   
  {\rmfamily} 
  {}          
  {\bfseries} 
  {.}         
  { }         
  {}          
\newtheoremstyle{propositional}%
  {\topsep}   
  {\topsep}   
  {\itshape}  
  {}          
  {\bfseries} 
  {.}         
  { }         
  {}          
\theoremstyle{propositional}
\newtheorem{thm}{Theorem}[section]
\newtheorem{prop}[thm]{Proposition}
\newtheorem{lem}[thm]{Lemma}
\newtheorem{cor}[thm]{Corollary}
\newtheorem*{Petrovskii}{Petrovski\u\i's criterion}
\theoremstyle{descriptive}
\newtheorem{definition}[thm]{Definition}
\newtheorem{remark}[thm]{Remark}
\newtheorem{example}[thm]{Example}
\newtheorem{openprob}[thm]{Open problem}

\makeatletter
\renewenvironment{proof}[1][\proofname]{\par
  \pushQED{\qed}%
  \normalfont
  \trivlist
  \item[\hskip\labelsep
        \itshape
    #1\@addpunct{.}]\ignorespaces
}{%
  \popQED\endtrivlist\@endpefalse
}
\makeatother
{\catcode`p =12 \catcode`t =12 \gdef\eeaa#1pt{#1}}      
\def\accentadjtext#1{\setbox0\hbox{$#1$}\kern   
                \expandafter\eeaa\the\fontdimen1\textfont1 \ht0 }
\def\accentadjscript#1{\setbox0\hbox{$#1$}\kern 
                \expandafter\eeaa\the\fontdimen1\scriptfont1 \ht0 }
\def\accentadjscriptscript#1{\setbox0\hbox{$#1$}\kern   
                \expandafter\eeaa\the\fontdimen1\scriptscriptfont1 \ht0 }
\def\accentadjtextback#1{\setbox0\hbox{$#1$}\kern       
                -\expandafter\eeaa\the\fontdimen1\textfont1 \ht0 }
\def\accentadjscriptback#1{\setbox0\hbox{$#1$}\kern     
                -\expandafter\eeaa\the\fontdimen1\scriptfont1 \ht0 }
\def\accentadjscriptscriptback#1{\setbox0\hbox{$#1$}\kern 
                -\expandafter\eeaa\the\fontdimen1\scriptscriptfont1 \ht0 }
\def\itoverline#1{{\mathsurround0pt\mathchoice
        {\rlap{$\accentadjtext{\displaystyle #1}
                \accentadjtext{\vrule height1.593pt}
                \overline{\phantom{\displaystyle #1}
                \accentadjtextback{\displaystyle #1}}$}{#1}}
        {\rlap{$\accentadjtext{\textstyle #1}
                \accentadjtext{\vrule height1.593pt}
                \overline{\phantom{\textstyle #1}
                \accentadjtextback{\textstyle #1}}$}{#1}}
        {\rlap{$\accentadjscript{\scriptstyle #1}
                \accentadjscript{\vrule height1.593pt}
                \overline{\phantom{\scriptstyle #1}
                \accentadjscriptback{\scriptstyle #1}}$}{#1}}
        {\rlap{$\accentadjscriptscript{\scriptscriptstyle #1}
                \accentadjscriptscript{\vrule height1.593pt}
                \overline{\phantom{\scriptscriptstyle #1}
                \accentadjscriptscriptback{\scriptscriptstyle #1}}$}{#1}}}}
\def\itunderline#1{{\mathsurround0pt\mathchoice
        {\rlap{$\underline{\phantom{\displaystyle #1}
                \accentadjtextback{\displaystyle #1}}$}{#1}}
        {\rlap{$\underline{\phantom{\textstyle #1}
                \accentadjtextback{\textstyle #1}}$}{#1}}
        {\rlap{$\underline{\phantom{\scriptstyle #1}
                \accentadjscriptback{\scriptstyle #1}}$}{#1}}
        {\rlap{$\underline{\phantom{\scriptscriptstyle #1}
                \accentadjscriptscriptback{\scriptscriptstyle #1}}$}{#1}}}}
\newdimen\extrawidth
\def\iintlim#1#2{\setbox0\hbox{$\scriptstyle#1$}%
	\setbox1\hbox{$\scriptstyle#2$}%
	\extrawidth=\wd1 \advance\extrawidth-\wd0
	\ifdim\extrawidth<0pt \extrawidth=0pt\fi%
	\int_{#1\kern\extrawidth \kern .5em}^{#2\kern -\wd1} \kern -.5em%
}
\numberwithin{equation}{section}
\newenvironment{ack}{\medskip{\it Acknowledgement.}}{}

\newcommand{\strutdepth}{\hbox{\vrule height0pt depth4.5pt width0pt}}
\newcommand{\strutheight}{\hbox{\vrule height9.5pt depth0pt width0pt}}

\DeclareMathOperator{\Div}{div}

\renewcommand{\phi}{\varphi}
\newcommand{\eps}{\varepsilon}
\newcommand{\al}{\alpha}
\newcommand{\de}{\delta}
\newcommand{\ka}{\kappa}

\newcommand{\R}{\mathbf{R}}
\newcommand{\Z}{\mathbf{Z}}
\newcommand{\Rn}{\mathbf{R}^n}
\newcommand{\simge}{\gtrsim}
\newcommand{\simle}{\lesssim}

\newcommand{\Thetat}{\widetilde{\Theta}}
\newcommand{\thetat}{{\tilde{\theta}}}
\newcommand{\kat}{{\tilde{\ka}}}
\newcommand{\ut}{\tilde{u}}

\newcommand{\p}{{$p\mspace{1mu}$}}

\DeclareMathOperator{\capp}{cap}

\newcommand{\uP}{\itoverline{H}}
\newcommand{\lP}{\itunderline{H}}
\newcommand{\UU}{\mathcal{U}}%
\newcommand{\LL}{\mathcal{L}}%
\newcommand{\bdy}{\partial}
\newcommand{\bdry}{\partial}
\newcommand{\setm}{\setminus}
\renewcommand{\emptyset}{\varnothing}
\newcommand{\ga}{\gamma}

\newcommand{\la}{\lambda}
\newcommand{\wt}{\widetilde{w}}
\newcommand{\vt}{\tilde{v}}
\newcommand{\be}{\beta}
\newcommand{\Thetam}{\Theta_-}
\newcommand{\phit}{\widetilde{\phi}}
\newcommand{\os}[2]{\overset{#1}{#2}}
\begin{document}
%
%
\authortitle{Anders Bj\"orn and Jana Bj\"orn}
{Barriers, Barenblatt solutions and regularity \\ 
of soda can domains for the heat equation \\ 
and nonlinear \p-parabolic equations}
{Barriers and regularity of soda can domains
for  \p-parabolic equations}

\author{
Anders Bj\"orn \\
\it\small Department of Mathematics, Link\"oping University, SE-581 83 Link\"oping, Sweden\\
\it \small anders.bjorn@liu.se, ORCID\/\textup{:} 0000-0002-9677-8321
\\
\\
Jana Bj\"orn \\
\it\small Department of Mathematics, Link\"oping University, SE-581 83 Link\"oping, Sweden\\
\it \small jana.bjorn@liu.se, ORCID\/\textup{:} 0000-0002-1238-6751
}

\date{}
\maketitle

\noindent {\small {\bf Abstract}.
In this paper we study when the 
origin $(0,0)$ is a regular (or
irregular) boundary point for the
so-called soda can domains of the type
\[
\Theta_{l,\theta}:= \{(x,t) \in \R^{n+1}: 0<-t < \theta |x|^l <\theta\},
\quad \text{with $l,\theta >0$,}
\]
for the \p-parabolic equation
$\partial_t u- \Delta_p u=0$, where $1<p<\infty$. 
For $p<2n/{(n+1)}$ 
and for  the heat equation (i.e.\ $p=2$) we completely 
determine when  the origin is regular for soda can domains.
The domains $\Theta_{l,\theta}$ have \emph{nonconvex} time sections 
with power dependence on time.
For domains 
with rotationally symmetric convex time sections with power dependence on time,
the regularity of the origin as the last point was 
characterized by Petrovski\u\i\ (in 1935) for the heat equation,
and almost completely  in the nonlinear case ($p \ne 2$)
in our earlier paper (joint with Gianazza, 
\emph{Math. Ann.} {\bf 368} (2017), 885--904).
}

\medskip

\noindent {\small \emph{Key words and phrases}:
barrier, 
boundary regularity,
heat equation,
Perron method,
\p-Laplacian, 
\p-parabolic equation,
soda can domain.

\medskip

}


\section{Soda cans}

In 1935, Petrovski\u\i~\cite{Petro2} proved the following result
about regularity of the last point.

\begin{Petrovskii}
The origin $(0,0)$ is regular for the heat equation
\[
\bdry_t u-\Delta u=0    
\]
with respect to the domain
\begin{equation*} 
 \{(x,t) \in \Rn\times(-1,0):
        |x|^2<-Kt \log |{\log(-t)}| \}
\end{equation*}
if and only if\/ $0<K\le4$.
\end{Petrovskii}

A boundary point $\xi_0$ is \emph{regular}
if for every continuous boundary data, the corresponding
solution to the Dirichlet problem 
attains the boundary data as a limit at $\xi_0$,
see Definition~\ref{def:regular}.

In particular, Petrovski\u\i's criterion implies that $(0,0)$ is regular for the domain
\begin{equation}   \label{eq-Petrovskii-set}
\{(x,t)\in \Rn \times (-1,0): -t > \theta |x|^l\},
\quad 
\text{with $l,\theta >0$,}
\end{equation}
if and only if  $0 < l \le 2$.

Parabolic equations in noncylindrical domains describe heat transfer, motion of fluids,
diffusion phenomena and pattern building in domains that change in time, with
applications in physics, chemistry and biology, 
as described in the survey article Knobloch--Krechetnikov~\cite{KnoKre15}.
Such time-changing problems have attracted
a lot of attention in the past years, see e.g.\ the following papers
and  the references therein:
Abdulla~\cite{abdulla03}, \cite{abdulla05},
Bauer~\cite{Bauer62},
Effros--Kazdan~\cite{EfKaz}, Evans--Gariepy~\cite{EvGa},
Lanconelli~\cite{Lanc73},
Landis~\cite{Landis69}, \cite{Landis91},
Petrovski\u\i~\cite{Petro2},
Pini~\cite{pini1}, \cite{pini2} ($n=1$) 
and Watson~\cite{watson78}, \cite{watson}
for the heat equation, and
Bertsch--Dal Passo--Franchi~\cite{BeDPaFr}, 
B\"ogelein--Duzaar--Scheven--Singer~\cite{BogDuzScheSing18},
Byun--Wang~\cite{ByunWang05}, 
Calvo--Novaga--Orlandi~\cite{CaNoOr17},
Guesmia--Harkat~\cite{GueHar18},  
Hofmann--Lewis~\cite{HofLewis96},
Kilpel\"ainen--Lindqvist~\cite{KiLi96}  
and Moring--Scheven--Sch\"atzler~\cite{MSS}
 for nonlinear parabolic  equations.
See also Section~\ref{sect-real-world} for two 
simple  real-world interpretations of soda can domains.
The 
general theory of nonlinear parabolic equations has been developed in the book by 
DiBenedetto~\cite{DiBen93},
see also the book by
DiBenedetto--Gianazza--Vespri~\cite{DiBenGianVesp12}.

In this paper we consider  the  \p-parabolic equation
\begin{equation} \label{eq:para}
\partial_t u- \Delta_p u
    :=\frac{\partial u}{\partial t}- \Div(|\nabla u|^{p-2} \nabla u)
    =0
\quad \text{in } \R^{n+1},
\quad p>1,
\end{equation}
both in the \emph{degenerate} case $p>2$ and the \emph{singular} case $1<p<2$.
It is a nonlinear generalization of the usual heat equation.
The heat equation (i.e.\ $p=2$) is also included in our study.

For the \p-parabolic equation a natural counterpart of
the Petrovski\u\i\ criterion is to study regularity of the origin for 
the rotationally symmetric domains~\eqref{eq-Petrovskii-set}.
Regularity in this case 
was almost completely (except for the case $1<l=p<2$)
characterized
in our paper~\cite[Theorem~1.1]{BBG} (joint with Gianazza).
Lindqvist~\cite{lindqvist95} 
and Bj\"orn--Bj\"orn--Gianazza--Par\-vi\-ain\-en~\cite[Sections~6--7]{BBGP}
contain 
some earlier results in this direction.
Incidentally, \cite[Theorem~1.1]{BBG} also shows that a \p-parabolic
analogue of Effros--Kazdan's tusk condition~\cite{EfKaz}
is not sufficient for regularity when $p>2$.
Moreover, unlike for $p=2$, regularity 
for the origin
in \eqref{eq-Petrovskii-set} holds if and only if
$0<l<p$ when $p>2$.

In this paper we instead consider the complementary
domains of the form
\begin{equation} \label{eq-soda}
    \Theta_{l,\theta} = \{(x,t) \in \R^{n+1}: 0<-t < \theta |x|^l <\theta\}, \quad
\text{with $l,\theta >0$.}
\end{equation}
As in \cite{BBGP}, we call such domains \emph{soda cans} as they (for $l>1$)
resemble the concave inside of the bottom
of a soda (or beer) can.
Just like the Petrovski\u\i\ set~\eqref{eq-Petrovskii-set},
the soda can domains~\eqref{eq-soda}
are rotationally symmetric domains of power type
(but with nonconvex time sections, namely annuli with shrinking
inner radii so that they converge to a punctured ball).
Despite this, they   
do not seem to have been studied much.

In contrast to Petrovski\u\i's criterion for the heat equation,
it was shown in 
Bj\"orn--Bj\"orn--Gianazza~\cite[Proposition~4.1]{BBG}
that   boundary regularity of the origin is independent of 
radial scaling in the $x$-direction 
(or equivalently  in the $t$-direction) when $p \ne 2$.
In particular, the regularity
for $\Theta_{l,\theta}$
(and for the Petrovski\u\i\ set~\eqref{eq-Petrovskii-set}) 
is independent of $\theta$ 
when $p \ne 2$, see Corollary~\ref{cor-indep-theta}.

Note that since regularity is independent of the future 
(see Theorem~\ref{thm-Omminus}),
it is natural to consider regularity
of the origin with respect to domains (or open sets) in 
the lower half-space.
Since regularity is also  a local property (see Section~\ref{sect-bdy-reg}),
the lower bound $t > -\theta$
(or the upper bound $|x|< 1$)
is inessential.

In the positive direction, we obtained 
the following result 
in~\cite{BBGP} when $p \ne 2$.
For $p=2$ it follows from the tusk condition
of Effros--Kazdan~\cite{EfKaz}
(which can also be found in the book by Watson~\cite[Theorem~8.52]{watson}).

\begin{thm} \label{thm-soda-reg}
\textup{(\cite[Proposition~4.2]{BBGP} and~\cite{EfKaz})}
Assume that $l\ge p/(p-1)$ if $1<p \le 2$, and $l>p$ if $p>2$.
Then $(0,0)$  is regular with respect to $\Theta_{l,\theta}$ 
for every $\theta>0$.
\end{thm}

For $1<p<2$, we improve the range of $l$'s in Theorem~\ref{thm-soda-reg}
to match $p>2$ as follows.

\begin{prop}   \label{prop-l>p-reg}
Assume that $1<p<2$ and $l>p$.
Then $(0,0)$  is regular with respect to $\Theta_{l,\theta}$ 
for every $\theta>0$.
\end{prop}

In the special case $p>n$, we have a further   
improvement of  Theorem~\ref{thm-soda-reg}.
In particular,  this shows that all soda cans
are regular for all $p$ when $n=1$, see also Remark~\ref{rmk-n=1}.

\begin{prop} \label{prop-soda-p>n}
If $p>n$ and $l,\theta >0$, then 
$(0,0)$  is regular with respect to $\Theta_{l,\theta}$.
\end{prop}

In Kilpel\"ainen--Lindqvist~\cite[p.~675]{KiLi96} it is claimed that an exterior ball
touching a boundary point from the past
(at the north pole of the ball) 
is sufficient for regularity
for all $p>1$
and, in particular, that the function $w$ constructed in~\cite[p.~674]{KiLi96} is a barrier
for large $\al$.
However, the argument in~\cite[p.~674--675]{KiLi96} assumes that $\de>0$ in that construction, which is not
the case when the contact point is the north pole.
On the other hand, for 
$p>2$ and sufficiently small $\al>0$, 
the same function $w$ from~\cite[p.~674]{KiLi96} 
provides one (local) barrier since 
\[
\Delta_p w \le (2\al)^{p-1}(n +p-2) e^{-(p-1)\al R^2} 
\le 2\al e^{-\al R^2} \tfrac12 R_0 \le \partial_t w. 
\]
Unfortunately, one barrier by itself does not guarantee regularity, cf.\ Theorem~\ref{thm:barrier-char}.
In fact, at least when $1<p<2$, one barrier can exist even at irregular boundary points, as shown in 
Bj\"orn--Bj\"orn--Gianazza~\cite[Proposition~1.2]{BBG}.
Neither the construction in~\cite[p.~674]{KiLi96}  nor its modification in \cite[Proposition~4.1]{BBGP}
gives a barrier family 
when the exterior ball touches the boundary point from the past
(at the north pole)
and therefore
does not imply regularity in that geometric situation.

For $p=2$, it follows from the tusk condition
of Effros--Kazdan~\cite{EfKaz} that 
the north pole exterior ball condition implies regularity.
This corresponds to $l=2$ in \eqref{eq-soda}.
The soda can domains, considered in this paper, thus
satisfy an exterior generalized paraboloid condition from the past.
For the heat equation ($p=2$) our regularity results in 
Propositions~\ref{prop-irreg-EG}\ref{reg-n=2} and~\ref{prop-Phi-rho}
are stronger than the north pole exterior ball condition.
In particular, any power-type cusp from the past is sufficient 
for regularity when $n=2$.
Proposition~\ref{prop-l>p-reg} implies that the north pole exterior ball condition 
is sufficient for regularity when $1<p<2$.

In the negative direction, the following result shows
irregularity for some $l>0$ when
$1<p<2\le n$.  
We are not aware of any irregularity results for $p>2$ for 
soda can domains.

\begin{prop} \label{prop-soda-irreg}
Assume that $1<p<2\le n$ and 
\begin{alignat*}{2}
l &\le p, &\quad&  \text{if }p<\frac{2n}{n+1}, \\
    l &< \frac{(n-p)(2-p)}{p-1},
    &\quad & \text{if } p \ge  \frac{2n}{n+1}.
\end{alignat*}
Then $(0,0)$  is irregular with respect to 
$\Theta_{l,\theta}$  for every $\theta>0$.
\end{prop}

Note that $l \le p$ for both ranges of $p$ in Proposition~\ref{prop-soda-irreg}.
The range $p>2n/(n+1)$ is exactly those $p$,
for which $\la>0$ in the Barenblatt solution, which therefore exists (when $p \ne 2$).
The exponent $p=2n/(n+1)$ is also a known critical exponent
for various types of behaviour for the \p-parabolic equation,
see e.g.\ DiBenedetto--Gianazza--Vespri~\cite[Section~6.21.4]{DiBenGianVesp12}.
We do not know if it is a critical exponent also
for the regularity of soda cans, or just an artefact of our proof.
The fact that $l\to0$ as $p\nearrow2$ in Proposition~\ref{prop-soda-irreg},
while all $0<l<p$ give irregularity in Proposition~\ref{prop-irreg-EG} 
for $p=2$ and $n \ge 3$, might suggest
that our proof is not optimal.

For the heat equation (i.e.\ $p=2$) we have
the following complete characterization when $n \ge 2$.
In Proposition~\ref{prop-Phi-rho}
we give a similar characterization for generalized soda cans
(not necessarily of power type) in $\Rn$, $n \ge 3$.

\begin{prop}  \label{prop-irreg-EG}
Assume that $p=2$  and $\theta,l>0$.
\begin{enumerate}
\item \label{reg-n=2}
If $n=2$, then $(0,0)$  is regular for the heat equation with respect to 
$\Theta_{l,\theta}$  for all $l,\theta>0$, but it is  irregular for 
the punctured cylinder $(B(0,r)\setm\{0\}) \times (-1,0)$,
where $B(0,r)=\{x \in \R^2 : |x|<r\}$.
\item \label{reg-n>=3}
If $n \ge 3$, then $(0,0)$  is regular for the heat equation with respect to 
$\Theta_{l,\theta}$  if and only if $l \ge 2$.
\end{enumerate}
\end{prop}

In order to prove Proposition~\ref{prop-irreg-EG}
(and~\ref{prop-Phi-rho})
we will use the 
Wiener  criterion for the heat equation from 
Lanconelli~\cite{Lanc73} (necessity) and
Evans--Gariepy~\cite{EvGa} (sufficiency).
There
is also a different Wiener criterion
for the heat equation
by  Landis~\cite{Landis69}, \cite{Landis91}.
Nevertheless, none of these Wiener type criteria 
seems to have been used much for concrete domains.
For $p\ne2$, no Wiener type criterion is known.

For $p=2$ (i.e.\ the heat equation), Abdulla~\cite[Theorem~1.1]{abdulla05} 
characterized the regularity of the last point
for arbitrary bounded rotationally symmetric domains with convex
time sections (i.e.\ the time sections are balls), and thereby generalized Petrovski\u\i's criterion.
Note that  the time sections of soda cans are rotationally
symmetric \emph{nonconvex} open sets (annuli), 
and they are therefore not covered by
Abdulla's result.
Another sufficient condition 
 in terms of an exterior cusp 
(different from the tusk condition) was given in Abdulla~\cite{abdulla03},  
but it also does not apply to the soda can domains.

For $p>2$ and $p/(p-1) \le l \le p$ we have only been able
to show regularity for small boundary data in the following form.
We include $l >p$ in the formulation, which leads to another proof 
of the case $p>2$ in Theorem~\ref{thm-soda-reg}, 
cf.\ Corollary~\ref{cor-barrier-family-Barenblatt}.
See Definition~\ref{def-barrier} and
Remark~\ref{rmk-trad-barrier} for the 
definition and 
role of barriers.

\begin{prop}  \label{prop-small-reg}
Assume that $p>2$.
Let $l\ge p/(p-1)$ and $\theta>0$. 
Then $(0,0)$ is a regular boundary point for 
$\Theta_{l,\theta}$ with small boundary data.
Namely,  if $0<\de\le 1$ and 
\begin{equation}  \label{eq-bound-f-ka}
|f-f(0,0)| \le \frac{p-2}{p} \biggl( \frac{\de^{\frac{p}{p-1}-l}}{np\theta} \biggr)^{1/(p-2)} 
\min\{|x|,\de\}^{p/(p-1)}
\end{equation}
on $\bdy \Theta_{l,\theta}$, then
\begin{equation} \label{eq-lim-Hf}
\lim_{\Theta_{l,\theta} \ni \xi \to (0,0)} \lP_{\Theta_{l,\theta}} f(\xi)  
=  \lim_{\Theta_{l,\theta} \ni \xi \to (0,0)} \uP_{\Theta_{l,\theta}} f(\xi) 
= f(0,0).
\end{equation}

Moreover, there is a barrier for $\Theta_{l,\theta}$ at $(0,0)$.
\end{prop}

Note that when $l \ge 2$,  the exterior ball condition 
is satisfied from the past. 
However, this is not
enough to imply regularity, 
see the discussion after Theorem~\ref{thm-soda-reg} above.

We actually give two proofs of Proposition~\ref{prop-small-reg},
one based on power type barriers and one on Barenblatt type barriers,
see Sections~\ref{sect-trad-bar}--\ref{sect-Baren}.
We have included the rather technical proof based on Barenblatt solutions
because it seems to be a natural approach, even though it does not give
a better result than the shorter proof in Section~\ref{sect-trad-bar}.
The fact that both  constructions give exactly the same bound
in \eqref{eq-bound-f-ka} for $p/(p-1) \le l \le p$
when $p>2$
(and thus also lead to the same bound $M_\theta$
for the partial regularity when $l=p>2$ in Corollary~\ref{cor-reg-l=p} below)
perhaps suggests that there might
be some deeper reason behind this restriction.
We have however not been able to find an example showing
irregularity for large boundary data when $p/(p-1) \le l \le p$  and $p>2$. 
Recall that for $l>p$, the origin is regular for $\Theta_{l,\theta}$ by
Theorem~\ref{thm-soda-reg}.

Note that for $l>p/(p-1)$, the coefficient in front of $\min$ in \eqref{eq-bound-f-ka}
increases as $\de\to0$ and hence allows ``steeper'' boundary data at $0$.
At the same time, smaller $\de$ put more of a restriction 
on the maximal
size of the boundary data (when $p/(p-1) \le l <p$).
For $l=p/(p-1)$, the coefficient in front of $\min$ in \eqref{eq-bound-f-ka}
is independent of $\de$ and the best choice is thus to take $\de=1$.
This immediately gives the following corollary.

\begin{cor}  \label{cor-reg-l=p/(p-1)}
Assume that $p>2$ and $\theta>0$.
If  $l= p/(p-1)$, then
\eqref{eq-lim-Hf} holds for every
boundary data satisfying
\begin{equation*} 
|f(x,t)-f(0,0)| \le M_\theta   |x|^{p/(p-1)}
\quad \text{for } (x,t) \in \bdy \Theta_{l,\theta},
\end{equation*}
where
\begin{equation}   \label{eq-M-theta}
M_\theta:=\frac {p-2}{p} \biggl(\frac{1}{np\theta}\biggr)^{1/(p-2)}.
\end{equation}
\end{cor}

For $l=p$,
we can treat any $f$ that is continuous at $(0,0)$ and 
bounded in terms of  $M_\theta$ as follows.

\begin{cor}  \label{cor-reg-l=p}
Assume that $l=p>2$ and $\theta>0$.
If $f$ is continuous at $(0,0)$ and satisfies
\begin{equation}   \label{eq-ass-with-Mtheta}   
 |f-f(0,0)| \le M_\theta
:=\frac {p-2}{p} \biggl(\frac{1}{np\theta}\biggr)^{1/(p-2)}
\quad \text{on } \bdy\Theta_{l,\theta},
\end{equation}
then \eqref{eq-lim-Hf} holds.
\end{cor}

Tables \ref{table-ex-1} and \ref{table-ex-2} summarize the known (ir)regularity results for 
soda cans with $n\ge2$.

\begin{table}[t]
\begin{center}
\begin{tabular}{|c||c|c|c|c|}
\hline
 $1<p\le 2 \le n$    & $p<\frac{2n\strutheight}{n+1\strutdepth}$ & $ \frac{2n}{n+1}\le p<2$ & $p=2<n$ &  $p=2=n$   \\  \hline  \hline
     $ 0<l<\min\bigl\{ p, \frac{(n-p)(2-p)\strutheight}{p-1\strutdepth} \bigr\}$ 
& irreg & irreg & 
$\emptyset$ & $\emptyset$ \\ \hline
     $ \frac{(n-p)(2-p)\strutheight}{p-1\strutdepth} \le l<p $ & 
$\emptyset$ & ? & irreg & reg\\ \hline
     $l=p$\strutheight\strutdepth & irreg & ? & reg & reg\\ \hline
     $l>p$\strutheight\strutdepth & reg & reg & reg & reg\\ \hline
\end{tabular}
\end{center}
\caption{This table summarizes the 
known (ir)regularity results for soda cans
with  $1<p\le 2 \le n$.
The symbol $\emptyset$
indicates that there are no $l$ in that range.
}

\label{table-ex-1}
\end{table}

\begin{table}[ht]
\begin{center}
\begin{tabular}{|c|c|c|c|c|}
\hline
\multicolumn{5}{|c|}{$2<p<n$\strutheight\strutdepth} \\ \hline
$ 0<l<\frac{p\strutheight}{p-1\strutdepth}$ & $ l=\frac{p}{p-1}$ & $ \frac{p}{p-1}<l<p$ &  $l=p$  & $l>p$ \\  \hline  \hline
 & \multicolumn{3}{|c|}{Partial regularity for (wlog $f(0,0)=0$):\strutheight\strutdepth} & \\ 
? &  $ |f(x,t)|\le M_\theta |x|^{\frac{p}{p-1}\strutheight}$ & $f$ as in \eqref{eq-bound-f-ka} & $|f(x,t)|\le M_\theta$  & reg \\ 
 &   &  & continuous at  $(0,0)$\strutdepth  & \\ \hline
\end{tabular}
\end{center}
\caption{This table summarizes the 
known (ir)regularity results for soda cans
with  $2<p<n$.
}
\label{table-ex-2}
\end{table}

We end this discussion by formulating the following open problems.

\begin{openprob}
Let $p>2$ and $\theta>0$. 
\begin{enumerate}
\item Is $(0,0)$ regular for $\Theta_{l,\theta}$ when $0 < l \le p$?
\item Is there a 
barrier at $(0,0)$ for $\Theta_{l,\theta}$ 
when $0 <l < p/(p-1)$?
\end{enumerate}
\end{openprob}

\begin{openprob}
Let 
\[
 \frac{2n}{n+1} \le p<2,\quad  \theta>0
\quad \text{and} \quad \frac{(n-p)(2-p)}{p-1} \le l\le p.
\]
Is $(0,0)$ regular for $\Theta_{l,\theta}$?
\end{openprob}

\begin{openprob}
Let $p=n=2$.
\begin{enumerate}
\item 
Can one find a precise formula for the capacity $\capp(C_{r,h})$ of
parabolic cylinders $C_{r,h}$ 
 when $h \ge r^2$? Cf.\ Lemma~\ref{lem-heat-cyl}. 
\item
Can such a formula (or something else) be used to 
give a full characterization 
of the regualarity of the origin
for generalized soda cans
(not necessarily of power type) in $\R^2$,
like the one for $n \ge 3$ in  Proposition~\ref{prop-Phi-rho}?
\end{enumerate}
\end{openprob}

The outline of the paper is as follows.
In Sections~\ref{sect-prelim} and~\ref{sect-bdy-reg} we introduce
some notation and necessary background results.
The proofs of Propositions~\ref{prop-l>p-reg}--\ref{prop-soda-irreg}
are  given in Section~\ref{sect-pf}.
In Section~\ref{sect-heat} we study the heat equation
and prove Proposition~\ref{prop-irreg-EG}.
In Section~\ref{sect-trad-bar} 
we give our first proof of Proposition~\ref{prop-small-reg}.
Also Corollary~\ref{cor-reg-l=p} is proved here, while
Corollary~\ref{cor-reg-l=p/(p-1)} is a direct consequence
of Proposition~\ref{prop-small-reg}.  
Section~\ref{sect-Baren} is devoted to some
regularity
results based on Barenblatt solutions,
including our second proof of Proposition~\ref{prop-small-reg}.
We end the paper with two 
simple  real-world interpretations of soda can domains.

\begin{ack}
A.~B. resp.\ J.~B. were supported by the Swedish Research Council,
  grants 2020-04011 and 2024-04095 resp.\ 2022-04048. 
\end{ack}

\section{Preliminaries}
\label{sect-prelim}

We will use the notation and several results from 
Bj\"orn--Bj\"orn--Gianazza--Par\-vi\-ain\-en~\cite{BBGP}.
Here we will be brief and only introduce and discuss what we really need,
see \cite{BBGP} for a more extensive discussion.

\medskip

\emph{From now on we will always assume that 
$\Theta\subset\R^{n+1}$ is a nonempty bounded open set
and $1<p< \infty$.}

\medskip

Let $U\subset \R^n$ be a bounded open set.
The \emph{parabolic Sobolev space} $L^p(t_1,t_2;W^{1,p}(U))$,
with $t_1<t_2$, is the space of functions $u(x,t)$ such that the mapping $x
\mapsto u(x,t)$ belongs to $W^{1,p}(U)$ for almost every $t_1 < t <
t_2$ and the norm
\[
\biggl(\iintlim{t_1}{t_2}\int_U (|u(x,t)|^{p} + |\nabla
u(x,t)|^{p})\,dx\,dt\biggr)^{1/p}
\]
is finite. 
Analogously, by the space $C([t_1,t_2];L^p(U))$,
with $t_1<t_2$, we mean the space of functions $u(x,t)$, such that the mapping
$t\mapsto\int_U|u(x,t)|^p\,dx$ is continuous in the time interval $[t_1,t_2]$.
(The gradient $\nabla$ is always taken with
respect to the $x$-variables in this paper.)

\begin{definition}
A function $u:\Theta \to [-\infty,\infty]$ is a  
(weak) \emph{solution} to the equation \eqref{eq:para} 
if whenever $U_{t_1,t_2}:=U \times(t_1,t_2) \Subset \Theta$ is an open cylinder, 
we have 
$u \in C([t_1,t_2];L^2(U))\cap L^{p}(t_1,t_2;W^{1,p}(U))$ 
and  $u$ satisfies the integral equality
\begin{equation} \label{eq-def-soln} 
\iintlim{t_1}{t_2}\int_{U} |\nabla u|^{p-2} \nabla u \cdot
\nabla\phi \, dx\,dt - \iintlim{t_1}{t_2}\int_{U} u
\frac{\partial\phi}{\partial t} \, dx\,dt   =  0
\end{equation}
for all $\phi \in C_0^\infty(U_{t_1,t_2})$.
A \emph{\p-parabolic function} is a continuous solution.

A function $u$
is a (weak) \emph{supersolution}
if whenever $U_{t_1,t_2} \Subset
\Theta$ we have $u \in
L^{p}(t_1,t_2;W^{1,p}(U))$ and the 
left-hand side in \eqref{eq-def-soln} 
is nonnegative for all
nonnegative $\phi \in C_0^\infty(U_{t_1,t_2})$.
\end{definition}

\begin{definition}\label{def:superparabolic}
A function $u:\Theta\rightarrow (-\infty,\infty]$
is \emph{\p-super\-parabolic} if
\begin{enumerate}
\renewcommand{\theenumi}{\textup{(\roman{enumi})}}%
\item $u$ is lower semicontinuous,
\item $u$ is finite in a dense subset of ${\Theta}$,
\item \label{def-comp}
$u$ satisfies the following comparison principle on each space-time
  box $Q_{t_1,t_2}\Subset{\Theta}$: If $h$ is \p-parabolic in
  $Q_{t_1,t_2}$ and continuous on $\itoverline{Q}_{t_1,t_2}$, and if
  $h\leq u$ on the parabolic boundary  $\partial_p Q_{t_1,t_2}$, 
then $h\leq u$ in the whole $Q_{t_1,t_2}$.
\end{enumerate}

A function $v:\Theta\rightarrow [-\infty,\infty)$
is \emph{\p-sub\-parabolic} if $-u$ is \p-superparabolic.
\end{definition}

Here $Q_{t_1,t_2}$ is a \emph{space-time box} if it is of the form
$ Q_{t_1,t_2}=Q\times(t_1,t_2)$, where $Q=(a_1,b_1)\times\dots\times(a_n,b_n)$.
The  \emph{parabolic boundary} of $Q_{t_1,t_2}$ is
\[
\partial_p Q_{t_1,t_2}=(\itoverline Q\times \{t_1\})\cup(\partial Q\times (t_1,t_2]).
\]

The connection between \p-superparabolic functions and supersolutions 
is a delicate issue.
However, a continuous supersolution is \p-superparabolic
by the comparison principle 
of Korte--Kuusi--Par\-vi\-ain\-en~\cite[Lemma~3.5]{KoKuPa10}.

We will need the following parabolic comparison principle.
In this paper, $\bdy \Theta$ always denotes the entire
(topological) boundary of $\Theta$.

\begin{thm} \label{thm-parabolic-comparison}
\textup{(Parabolic comparison principle, \cite[Theorem~2.4]{BBGP})}
Suppose that $u$ is \p-super\-parabolic and $v$ is \p-sub\-parabolic
in $\Theta$.
Let $T \in \R$ and assume that
 \[
  \infty \ne    \limsup_{\Theta \ni (y,s)\rightarrow (x,t)} v(y,s)\leq
   \liminf_{\Theta \ni (y,s)\rightarrow (x,t)} u(y,s) \ne -\infty
 \]
for all $(x,t) \in \partial\Theta$ with $t< T$.
Then $v\leq u$ in $\{(x,t) \in \Theta : t<T\}$.
\end{thm}

\begin{lem} \label{lem-pasting}
\textup{(Pasting lemma,  \cite[Lemma~2.9]{BBGP})}
Let $G \subset \Theta$ be open. 
Also let $u$ and $v$ be \p-superparabolic in $\Theta$ and $G$,
respectively,
and let
\[
    w=\begin{cases}
     \min\{u,v\} & \text{in } G, \\
     u & \text{in } \Theta \setm G. \\
    \end{cases}
\] 
If $w$ is lower semicontinuous, then $w$ is \p-superparabolic in $\Theta$.
\end{lem}

We now turn to the Perron method.
It will be enough to consider Perron solutions for
bounded functions, so for simplicity we restrict ourselves
to this case.

\begin{definition}   \label{def-Perron}
Given a bounded function $f \colon \bdy \Theta \to \R$,
let the upper class $\UU_f$ be the set of all
\p-superparabolic  functions $u$ on $\Theta$ 
such that
\[
    \liminf_{\Theta \ni \eta \to \xi} u(\eta) \ge f(\xi) \quad \text{for all }
    \xi \in \bdy \Theta.
\]
Define the \emph{upper Perron solution} of $f$  by
\[
\uP_{\Theta} f (\xi) =
\inf_{u \in \UU_f}  u(\xi), \quad \xi \in \Theta.
\]
Similarly, let 
$\LL_f$ be the set of all
\p-subparabolic functions $u$ on $\Theta$ 
such that
\[
\limsup_{\Theta \ni \eta \to \xi} u(\eta) \le f(\xi) \quad \text{for all }
\xi \in \bdy \Theta,
\]
and define the \emph{lower Perron solution}  of $f$ by
\[
\lP_{\Theta} f (\xi) = 
\sup_{u \in \LL_f}  u(\xi), \quad \xi \in \Theta.
\]
\end{definition}

When the set is clear from the context, we often drop $\Theta$ from the
notation and simply write $\uP f$ etc.
It follows from the 
parabolic comparison principle (Theorem~\ref{thm-parabolic-comparison})
that 
\begin{equation} \label{eq-lp<=uP}
\lP f \le \uP f.
\end{equation}
Moreover $\lP f = - \uP (-f)$.
Kilpel\"ainen--Lindqvist~\cite[Theorem~5.1]{KiLi96}
showed that both $\lP f$ and $\uP f$
are \p-parabolic.

The following simple lemma is easily proved by direct calculation.
It also follows from Lemma~\ref{lem-Deltap-polar}.

\begin{lem}   \label{lem-Delta-powers}
\textup{(Bj\"orn--Bj\"orn--Gianazza~\cite[Lemma~2.5]{BBG})}
For any $\al, C \in\R$ and $x\ne0$ we have
\[
\Delta_p (C|x|^\al) = C\al |C\al |^{p-2}
(n+(\al-1)(p-1)-1) |x|^{(\al-1)(p-1)-1}.
\]
In particular, 
$\Delta_p (C|x|^{p/(p-1)}) = (Cp/(p-1))^{p-1}n$ when $C\ge0$.
\end{lem}

\section{Boundary regularity}
\label{sect-bdy-reg}

\begin{definition}
\label{def:regular}
A boundary point $\xi_0\in \partial \Theta$ is 
\emph{regular} with respect to $\Theta$, if
\[
        \lim_{\Theta \ni \xi \to \xi_0} \uP f(\xi)=f(\xi_0)
\quad \text{for all }f \in C(\bdy \Theta).
\]
\end{definition}

Since $\lP f = - \uP (-f)$, regularity can equivalently
be formulated using lower Perron solutions.
Note that, for $p \ne 2$ and  general $\Theta$, it is not known
whether $\lP f = \uP f$, even for continuous $f$.
Next we summarize some of the main results from 
Bj\"orn--Bj\"orn--Gianazza--Par\-vi\-ain\-en~\cite{BBGP},
which will be needed later.

\begin{definition}
\label{def-barrier} 
Let $\xi_0\in \partial \Theta$. 
A positive \p-superparabolic function $w$ in $\Theta$
is a \emph{barrier} at $\xi_0$ if 
\begin{equation} \label{eq-trad-barrier} 
\lim_{\Theta \ni \zeta\to \xi_0} w(\zeta)=0 
   \quad \text{and} \quad
  \liminf_{\Theta \ni\zeta\to \xi} w(\zeta)>0
   \text{ for all } \xi \in \bdy \Theta \setm \{\xi_0\}.
\end{equation}

A \emph{barrier family} at $\xi_0$ is a family of 
barriers $w_j$, $j=1,2,\dots$\,,  such that 
for each $k=1,2,\dots$\,, there is a $j$ for which
\[
   \liminf_{\Theta \ni \zeta\to\xi} w_j(\zeta) \ge k
   \quad \text{for every } \xi \in \bdy \Theta 
   \text{ with }|\xi-\xi_0| \ge 1/k.
\]
\end{definition}

This definition of barrier family lies between
the barrier family and the strong barrier family 
from \cite{BBGP}, and thus it also characterizes regularity
in the following way.

\begin{thm}\label{thm:barrier-char}
\textup{(\cite[Theorem~3.3]{BBGP})} 
A boundary point  $\xi_0 \in \bdy \Theta$
is regular if and only if it has a barrier family.
\end{thm}

It follows from 
Theorem~\ref{thm:barrier-char} and the pasting Lemma~\ref{lem-pasting}
that regularity is a local property of the boundary and that it
is inherited by subsets
(see~\cite[Proposition~3.4]{BBGP} for $p\ne2$ and Bauer~\cite[Theorems~28--29]{Bauer62}  for $p=2$).
Moreover, the following characterization holds.

\begin{thm} \label{thm-Omminus}
\textup{(\cite[Lemma~3.10 and Theorem~3.11]{BBGP})} 
A boundary point  $\xi_0=(x_0,t_0) \in \bdy \Theta$
is regular with respect to $\Theta$
if and only if either $\xi_0$ is regular with respect to 
$\Thetam:=\{(x,t) \in \Theta : t < t_0\}$
or $\xi_0 \notin \bdy \Thetam$.
\end{thm}

\begin{remark} \label{rmk-trad-barrier}
In classical potential theory,  
a \emph{barrier} is a superharmonic (when dealing with the 
Laplace equation) or superparabolic 
(when dealing with the heat equation) 
function $w$ satisfying \eqref{eq-trad-barrier}.

Existence of a single 
barrier implies the regularity of a 
boundary point in the classical linear theories,
since  one can scale and lift the barrier.
The same is true for \p-harmonic functions, i.e.\ solutions of
the nonlinear elliptic \p-Laplace equation
$\Delta_p u=0$, see 
Kilpel\"ainen~\cite[Theorem~1.5]{Kilp89} or
Heinonen--Kilpel\"ainen--Martio~\cite[Theorem~9.8]{HeKiMa}.
In fact, in these cases,
the existence of a barrier
is \emph{equivalent} to the regularity of $\xi_0$,
see Lebesgue~\cite{Lebesgue24} (for harmonic functions),
Bauer~\cite[Theorems~30 and~31]{Bauer62} (for the heat equation)
and \cite{HeKiMa} or \cite{Kilp89} (for \p-harmonic functions).

On the other hand, for the \p-parabolic equation with $1<p<2$, it was
shown in
Bj\"orn--Bj\"orn--Gianazza~\cite[Proposition~1.2]{BBG}
that the existence of one barrier 
is \emph{not} enough to guarantee regularity.

Our definition of barriers is the
same as in~\cite{Lebesgue24} and~\cite{HeKiMa}.
In the linear theory, 
weaker requirements on the barrier are often assumed,
see e.g.~\cite[Definition~9]{Bauer62},
but the existence of such a barrier is
equivalent to the existence of a barrier as defined above.
\end{remark}

The following scaling lemma shows that for $p\ne2$
it suffices to study boundary regularity 
with respect to $\Theta_{l,\theta}$ for one particular $\theta$ only.
It also makes it possible to stretch a barrier, 
defined only in a neighbourhood of
$(0,0)$, to the whole soda can domain~$\Theta_{l,\theta}$.
Another way of extending barriers to larger sets is by the pasting 
Lemma~\ref{lem-pasting}.
See Bj\"orn--Bj\"orn--Gianazza~\cite[Proposition~4.1]{BBG} for a similar result.

\begin{lem} \label{lem-scaling-gen}
Let $p\ne2$ and  $a,b,l,\de,\theta>0$.
For a function $u$, defined in
the set 
\[
\Theta := \{(x,t): 
0 < -t < \theta|x|^l
\text{ and } |x|<\de\},
\]
let 
\[
    \ut(x,t)= A u(ax,bt),
    \quad \text{with } A=\biggl( \frac{b}{a^p} \biggr)^{1/(p-2)},
\]
which is defined in 
\[
\Thetat := \{(x,t): 
0 < -t < \thetat |x|^l
\text{ and } |x|<\de/a\},
\quad \text{with } \thetat = \frac{a^l\theta}{b}.
\]
Then the following are true\/\textup:
\begin{enumerate}
\item \label{s-soln}
$u$ is a solution in $\Theta$ if and only
if $\ut$ is a solution in $\Thetat$.
\item \label{s-supersoln}
$u$ is a supersolution in $\Theta$ if and only
if $\ut$ is a supersolution in $\Thetat$.
\item \label{s-superparabolic}
$u$ is \p-superparabolic in $\Theta$ if and only
if $\ut$ is a \p-superparabolic in $\Thetat$.
\end{enumerate}

Consequently, $u$ is a barrier at $(0,0)$ for 
$\Theta$ if and only if $\ut$ is a barrier for $\Thetat$.
Similarly, there is a barrier family at $(0,0)$ for 
$\Theta$ if and only if there is a barrier family for $\Thetat$.

Moreover, for $f \in C(\bdy \Theta)$ and 
$\tilde{f}(x,t) :=A f(ax,bt)\in C(\bdy \Thetat)$, 
we have
\begin{equation} \label{eq-scaling-newer}
 \uP_{\Thetat}  \tilde{f}(x,t) 
  = A \uP_{\Theta} f(ax,bt).
\end{equation}

In particular, 
if $a=\de$ and $b=\de^l$ then $\Thetat =  \Theta_{l,\theta}$.
Similarly, $\Theta =\Theta_{l,\theta}$
and  $\Thetat=\Theta_{l,1}$  when $a=\de=1$ and
$b=\theta$.
\end{lem}

\begin{proof}
The change of variables $(x,t) \mapsto (x/a,b/t)$, 
and replacing $u$ by $\ut$ and $\phi$ by $    \phit(x,t):=  \phi(ax,bt)$,
multiplies the integrand in the first
double integral in~\eqref{eq-def-soln} by $A^{p-1}a^p=Ab$ and 
the integrand in the second double integral in~\eqref{eq-def-soln} by $Ab$.
Hence the sign of the left-hand side in~\eqref{eq-def-soln} remains the same.
This proves \ref{s-soln} and \ref{s-supersoln}.

Part~\ref{s-superparabolic} now follows since part~\ref{def-comp}
in Definition~\ref{def:superparabolic} remains invariant
under this change of variables, by~\ref{s-soln},
and the other two properties are trivially invariant.
The identity \eqref{eq-scaling-newer} then follows directly from the definition of
Perron solutions.
\end{proof}

\begin{cor} \label{cor-indep-theta}
Assume that $l,\theta_1,\theta_2>0$ and $p \ne 2$. 
Then $(0,0)$ is regular 
with respect to  $\Theta_{l,\theta_1}$ 
if and only if it is
regular 
with respect to  $\Theta_{l,\theta_2}$. 
\end{cor}

\begin{remark} \label{rmk-n=1}
When $n=1$, the ``soda can domain'' $\Theta_{l,\theta}$ is not a domain as 
it has two components, with 
the origin as a common boundary point.
Nevertheless, 
the regularity of the origin with respect to 
each component, and thus also with respect to $\Theta_{l,\theta}$,
follows e.g.\ from
the exterior ball condition 
(Proposition~4.1 
in Bj\"orn--Bj\"orn--Gianazza--Par\-vi\-ain\-en~\cite{BBGP}).
\end{remark}

\section{Proofs of Propositions~
\ref{prop-l>p-reg}--\ref{prop-soda-irreg}
}
\label{sect-pf}

\begin{proof}[Proof of Proposition~\ref{prop-l>p-reg}]
By Corollary~\ref{cor-indep-theta} and the comment after Theorem~\ref{thm:barrier-char},
we can assume that $0<\theta<1$ and $p<l<2$.
For integers $j>0$, let 
\[
   u_j(x,t)= j(|x|^\al - (-t)^\al),
   \quad \text{where $\al=1-\frac{p}{l}>0$.}
\]
Note that for $0<-t \le \theta |x|^l \le \theta$,
\begin{equation}    \label{eq-uj>0}
u_j(x,t) \ge j(|x|^{\al} - (\theta|x|^l)^\al) 
\ge j|x|^\al (1-\theta^{\al l}) > 0,
\end{equation}
since $l>1$ and $\theta^{\al l} <1$.
Moreover, 
\begin{equation} \label{eq-1.2}
\lim_{(x,t)\to(0,0)} u_j(x,t)=0.
\end{equation}
Lemma~\ref{lem-Delta-powers} implies that
\begin{align}
\partial_t u_j 
- \Delta_p u_j &= j\al (-t)^{\al-1} - (j\al)^{p-1}\de |x|^{(\al-1)(p-1)-1} \nonumber \\
    &= j\al (-t)^{\al-1} \bigl( 1 - (j\al)^{p-2}\de (-t)^{1-\al}|x|^{(\al-1)(p-1)-1} \bigr),
\label{eq-Cu}
\end{align}
where $\de=n+(\al-1)(p-1)-1$. 
Now, for $(x,t) \in \Theta_{l,\theta}$,
\[
(-t)^{1-\al}|x|^{(\al-1)(p-1)-1} < |x|^{l(1-\al)+(\al-1)(p-1)-1} \le 1,
\]
because $\theta <1$ and the exponent 
\[ 
l(1-\al)+(\al-1)(p-1)-1 
= l \frac{p}{l} + \al(p-1)-p
= \al(p-1)>0.
\]

Thus, using that $\al>0$ is fixed and $p-2<0$, we conclude from~\eqref{eq-Cu}
that $u_j$ is 
a continuous supersolution in $\Theta_{l,\theta}$ for 
all sufficiently large $j$.
By  the comparison principle in
Korte--Kuusi--Par\-vi\-ain\-en~\cite[Lemma~3.5]{KoKuPa10}, 
these $u_j$ are \p-superparabolic functions and it follows from \eqref{eq-uj>0}
and \eqref{eq-1.2} that they
form a barrier family for $\Theta_{l,\theta}$ at $(0,0)$.
Theorem~\ref{thm:barrier-char} then shows that $(0,0)$ is regular.
\end{proof}

In the proof of Proposition~\ref{prop-soda-p>n} we 
shall use the following
formula for the \p-Laplacian in polar coordinates.

\begin{lem} \label{lem-Deltap-polar}
Assume that $u,v \in C^1((0,\infty))$ and 
let\/ {\rm(}by abuse of notation\/{\rm)} 
$u(x)=u(|x|)$ and $v(x)=v(|x|)$ for $x \in\Rn \setm \{0\}$.
Then, for $x \in\Rn \setm \{0\}$,
\begin{equation}  \label{eq-n}
   \nabla u(x) = u'(|x|) \frac{x}{|x|}
\quad \text{and} \quad 
   \Div \biggl(v(x)\frac{x}{|x|}\biggr)  = v'(|x|) + v(x) \frac{n-1}{|x|}.
\end{equation}

If moreover, $u\in C^2((0,\infty))$ and $u' \ge 0$, then
\begin{equation} \label{eq-Deltap}
   \Delta_p(u(x))= u'(|x|)^{p-2} \biggl( (p-1) u''(|x|) + u'(|x|) \frac{n-1}{|x|}\biggr).
\end{equation}
\end{lem}

\begin{proof}
The formulas~\eqref{eq-n} are clear by direct calculation.
Thus, for \eqref{eq-Deltap}, we have with $v(x)= u'(|x|)^{p-1}$,
\begin{align*} 
   \Delta_p(u(x))&
= \Div \biggl(u'(|x|)^{p-1}  \frac{x}{|x|} \biggr) 
 = (p-1) u'(|x|)^{p-2}u''(|x|) + u'(|x|)^{p-1} \frac{n-1}{|x|}\\
&= u'(|x|)^{p-2} \biggl( (p-1) u''(|x|) + u'(|x|) \frac{n-1}{|x|}\biggr).
\qedhere
\end{align*}
\end{proof}

\begin{proof}[Proof of Proposition~\ref{prop-soda-p>n}]
Note that
\[
 \Theta_{l,\theta}  \subset 
   \Thetat  :=\{x : 0 < |x| < 1\} \times (-\theta,0)=: G \times (-\theta,0),
\]
where $\Thetat$ is a (punctured) cylinder and $G$ is a punctured ball.
Therefore, by the comments after Theorem~\ref{thm:barrier-char},
it suffices to show that the origin is regular for $\Thetat$.
We shall construct a barrier family $\{w_j\}_{j=1}^\infty$ in $\Thetat$
at $\xi_0=(0,0)$.
Fix a positive integer~$j$.

First, we aim to find a 
radially symmetric function $u_j\in C^2((0,\infty))$  
such that, with $u_j(x)=u_j(|x|)$ (by abuse of notation),
\[
\Delta_p u_j = -j 
\text{ in } G, \quad
u_j(0)=0
\quad \text{and} \quad  u_j(1)=a_j, 
\]
where $a_j \ge j$ will be chosen later.
By Lemma~\ref{lem-Deltap-polar} and assuming that $u_j'\ge0$,
\[
\Delta_p u_j = (u_j')^{p-2} \biggl( (p-1)u_j'' + \frac{n-1}{r} u_j' \biggr),
\]
where the radial derivatives are with respect to $r=|x|$.
Letting $v_j=u_j'$, we therefore want to solve the equation
\begin{equation}   \label{eq-p-Lapl-for-v}
v_j^{p-2} \Bigl( (p-1)v_j' + \frac{n-1}{r} v_j \Bigr) = -j. 
\end{equation}
With $z_j:=r^{(n-1)/(p-1)}v_j$, the left-hand side can be rewritten as
\[
(p-1) (r^{(n-1)/(p-1)}v_j)^{p-2} \frac{(r^{(n-1)/(p-1)} v_j)'}{r^{n-1}}
=\frac{(z_j^{p-1})'}{r^{n-1}}
\]
and so \eqref{eq-p-Lapl-for-v} becomes $(z_j^{p-1})' = -j r^{n-1}$.
We are interested in $z_j\ge0$. Hence
\[
z_j = \biggl( C_j - \frac{jr^n}{n} \biggr)^{1/(p-1)},
\quad \text{where}  \quad
C_j \ge  \frac{j}{n}. 
\]
It follows that
\[
u_j' = v_j = r^{(1-n)/(p-1)} z_j 
= \biggl( C_j r^{1-n} - \frac{jr}{n} \biggr)^{1/(p-1)}.
\]
Since we require $u_j(0)=0$ and $u_j(1)>0$, 
this yields a general radially symmetric solution
\[
u_j(x) = u_j(|x|) 
= \int_0^{|x|} \biggl( C_j \rho^{1-n} 
- \frac{j\rho}{n} \biggr)^{1/(p-1)}\,d\rho,
\]
where $C_j \ge j/n$
should be chosen so that it gives the correct value $u_j(1)=a_j$.
This is possible provided that
\[
   a_j \ge 
a_j':=\int_0^{1} \biggl( \frac{j\rho^{1-n}}{n} - \frac{j\rho}{n} \biggr)^{1/(p-1)}\,d\rho.
\]
Note that $a_j'<\infty$ since $p>n$.

We let $a_j=\max\{j,a_j'\}$ and choose $C_j$ accordingly.
Then $u_j$ is a \p-supersolution in $G$ (i.e.\ $\Delta_p u_j \le 0$),  while
$\phi_j(x):=j|x|$ 
is a \p-subsolution in $G$ by Lemma~\ref{lem-Delta-powers}
or~\ref{lem-Deltap-polar}.
Note that  $u_j(0)=\phi_j(0)=0$ 
and $u_j(1)=a_j \ge \phi_j(1)$. 
When $n \ge 2$, it follows from the comparison Lemma~3.18 in 
Heinonen--Kilpel\"ainen--Martio~\cite{HeKiMa} (and continuity)
that  $u_j \ge \phi_j$ in $G$.
The case $n=1$ is not considered in~\cite{HeKiMa},
but then
$u_j$ is concave on $(0,1)$ while $\phi_j$ is convex,
so $u_j \ge \phi_j$ in $G$ also in this case. 
Next define
\[
w_j(x,t)=u_j(x)-j t \ge 0
\quad \text{in } \Thetat.
\]
Then
\[
\Delta_p w_j
=-j=\partial_t w_j 
\quad \text{in } \Thetat, 
\]
and $w_j$ is therefore \p-parabolic in $\Thetat$. 
It follows that $\{w_j\}_{j=1}^\infty$
is a barrier family, and thus $\xi_0$ is regular with respect to $\Thetat$, 
by Theorem~\ref{thm:barrier-char}.
\end{proof}

\begin{remark}
A more explicit barrier family than the one constructed
above for $p > n$ can be based on the radially symmetric functions
\[
\ut(x)= \ut(|x|)= a|x|^{(p-n)/(p-1)} - b|x|^{p/(p-1)},
\quad \text{with suitable } a,b>0.
\]
We leave the details to the interested reader.
\end{remark}

\begin{remark}
The barrier family constructed in the proof of Proposition~\ref{prop-soda-p>n}
shows that any point in the lateral boundary of any parabolic cylinder 
(even noncircular) is regular if $p>n$.
(Using a simple scaling argument it applies also to larger cylinders.)
Hence, Theorem~6.5 in Kilpel\"ainen--Lindqvist~\cite{KiLi96} 
and Theorem~3.9 in~\cite{BBGP} hold for $p>n$.
(The proofs presented therein 
rely on the  existence result from the book 
Mal\'y--Ziemer~\cite[Theorem~6.21]{MaZi}, 
which implicitly assumes that $p\le n$,
and thus cannot be used for $p>n$.
For $p \le n$,  the proofs in \cite{KiLi96}  and \cite{BBGP} rely also
on the regularity
result \cite[Corollary~6.22]{MaZi},  while this is trivial for $p>n$.)
\end{remark}

\begin{proof}[Proof of Proposition~\ref{prop-soda-irreg}]
By Corollary~\ref{cor-indep-theta}, 
we may assume that $\theta=1$.
Consider the function
\[
   u(x,t)= C \biggl( \frac{-t}{|x|^l} \biggr)^{\be},  
\]
where $C>0$ will be chosen later and $\be:=1/(2-p)>0$.
First we shall
see that $u$ is \p-superparabolic in $\Theta_{l,1}$.

By Lemma~\ref{lem-Delta-powers} with $\al=-\be l$,
\[
   \Delta_p u = - C^{p-1}(-t)^{\be (p-1)} (\be l) ^{p-1} \de |x|^{-{\be l (p-1) -p}},
\]
where $\de:= n- (\be l +1)(p-1) -1 >0$ 
by the choice of $l$.
Also
\[
    \partial_t u  = - C\be (-t) ^{\be -1} |x|^{-\be l}. 
\]
Thus, using that $\be(p-2)=-1$,
together with
$|x|^{l-p}\ge1$ in $\Theta_{l,1}$ (because $l \le p$), we get
\begin{align*}
  \partial_t u - \Delta_p u 
  & = C(-t) ^{\be -1} |x|^{-\be l}
     (-\be + C^{p-2}(-t)^{\be(p-2) +1} \de (\be l)^{p-1} |x|^{-\be l (p-2) -p}) \\
  & = C(-t) ^{\be -1} |x|^{-\be l} 
     (-\be + C^{p-2} \de (\be l) ^{p-1} |x|^{l -p}) \\
  &    \ge  C(-t) ^{\be -1} |x|^{-\be l} 
     (-\be + C^{p-2} \de (\be l)^{p-1}),
\end{align*}
which is $0$ if we choose $C>0$ so that the last bracket becomes
$0$ (which is possible).
Thus $u$ is a continuous supersolution in $\Theta_{l,1}$ and 
is therefore \p-superparabolic therein, by  the comparison principle in
Korte--Kuusi--Par\-vi\-ain\-en~\cite[Lemma~3.5]{KoKuPa10}.

Now let  $f \in C(\bdy \Theta_{l,1})$ be defined by
\[
   f(x,t)=\begin{cases}
     u(x,t), & \text{if } t <0, \\
     C(1-|x|), & \text{if } t=0,
     \end{cases}
\]
which indeed is a continuous function on $\bdy \Theta_{l,1}$.

If $v \in \LL_f$, then the parabolic comparison principle
(Theorem~\ref{thm-parabolic-comparison})
shows that $v \le u$ in $\Theta_{l,1}$. 
Taking the supremum over all such $v$ gives $0 \le \lP f \le u$.
Thus
\[
    0 \le \liminf_{t \to 0-} \lP f(x,t)
    \le 
    \liminf_{t \to 0-}  u(x,t)
    = 0, \quad \text{if } 0 < |x| < 1,
\]
and so
\[
    \liminf_{\Theta_{l,1} \ni (x,t) \to (0,0)} \lP f(x,t)
    = 0 
    < f(0,0),
\]
which shows that $(0,0)$ is irregular for $\Theta_{l,1}$.
\end{proof}

\section{The heat equation and Proposition~\ref{prop-irreg-EG}}
\label{sect-heat}

In order to prove Proposition~\ref{prop-irreg-EG},
we will use the 
Wiener  criterion from 
Lanconelli~\cite[Corollario~1.2, with $\la=\tfrac12$]{Lanc73} and
Evans--Gariepy~\cite[Theorem~1, (1.11) with $\la=2$]{EvGa}.
It  shows that 
the origin is regular for the heat equation in
$\Theta$, with $(0,0) \in \bdy \Theta$,
if and only if
\begin{equation}  \label{eq-Wiener-EG}
\sum_{k=1}^\infty 2^k \capp \bigl( \itoverline{A_k \setm A_{k+1}}\setm 
\Theta\bigr) = \infty,
\end{equation}
where $\capp$ denotes the thermal (or parabolic) capacity
(see \eqref{deff-capp} below),
\[
A_k = \{(x,t) \in \R^n \times (-\infty,0): F(x,-t) \ge 2^{k}\},
\]
 and 
\[
F(x,t) = \biggl( \frac{1}{4\pi t} \biggr)^{n/2} e^{-|x|^2/4t}, 
\quad t>0,
\]
is the fundamental solution for the heat equation.
Lanconelli~\cite{Lanc73} showed the necessity of the Wiener criterion
(i.e.\ that a point is irregular if the
sum in \eqref{eq-Wiener-EG} converges). 
The sufficiency was later shown by Evans and Gariepy~\cite{EvGa}.

The thermal capacity $\capp$ 
is defined by means of parabolic potentials.
Namely, for a compact set $K$,
\begin{equation} \label{deff-capp}
\capp(K)= \sup \mu(K),
\end{equation}
where the supremum is taken over all measures $\mu$ supported on $K$
with 
\begin{equation}  \label{eq-par-pot}
P^\mu(x_0,t_0):=
\int_{\R^n\times(-\infty,t_0)} \frac{e^{-|x-x_0|^2/4(t_0-t)}}{(4\pi(t_0-t))^{n/2} } \, d\mu(x,t)
\le 1
\end{equation}
for all  $(x_0,t_0)\in \R^{n+1}$.

The following estimates for capacities of cylinders will
be crucial when applying the Wiener criterion.
Below, we write $B(0,r)=\{x \in \Rn : |x|<r\}$.
We also write $a \simle b$ and $b \simge a$
if there is an implicit comparison constant $C>0$ such that $a \le Cb$, 
and $a \simeq b$ if $a \simle b \simle a$.

\begin{lem} \label{lem-heat-cyl}
Consider the parabolic cylinder
$C_{r,h}=\itoverline{B(0,r)}\times [0, h]$.
Then
\begin{alignat}{2}
    \capp(C_{r,h}) & \simeq r^n, &\quad& \text{if } 0 \le h \le r^2, 
\label{eq-heat-cyl-1}\\
 \capp(C_{r,h}) & \simeq h r^{n-2}, &\quad& \text{if } h \ge r^2 \text{ and } n\ge3, 
\label{eq-heat-cyl-n>=3} \\
\frac{h}{1+\log
(h/r^2)} \simle   
 \capp(C_{r,h}) & \simle h, &\quad& \text{if } h \ge r^2 \text{ and } n=2.
\label{eq-heat-cyl-n=2}
\end{alignat}
\end{lem}

The formula~\eqref{eq-heat-cyl-1} was shown by
Lanconelli~\cite[Proposizione~5.1]{Lanc73}.
Note that the upper bound in~\eqref{eq-heat-cyl-n>=3}  is better
than the upper bound in \cite[Proposizione~6.1]{Lanc73} when $n\ge 3$.
In Avelin--Kuusi--Par\-vi\-ain\-en~\cite[Theorem~4.9 and
Corollary~2]{AvelinKP15}, 
the formulas \eqref{eq-heat-cyl-1}--\eqref{eq-heat-cyl-n>=3}
were obtained
for the corresponding nonlinear \p-parabolic capacity when $2 \le p <n$, including
the case $p=2$. 
Rather than letting the reader study this sophisticated nonlinear paper, we give a
more elementary proof only covering $p=2$, which is enough for our purposes.

\begin{proof}
It is well known and easy to see
(cf.\ \cite[Proposizione~4.1(ii)]{Lanc73})
that 
\[
\capp (C_{r,r^2}) \simeq r^n.
\]
On the other hand, 
by \cite[Proposizione~5.1]{Lanc73} (or Watson~\cite[Theorem~9]{watson78} or
\cite[Theorem~7.55]{watson}),
\[
\capp (C_{r,0}) \simeq r^n.
\]
Since $\capp$ is a monotone set function, this shows \eqref{eq-heat-cyl-1}.
(It follows from \cite[Lemma~4]{watson78} or
\cite[Lemmas~7.37 and~7.54]{watson} that the thermal 
capacity in \cite{watson78} and \cite{watson} (with $E=\R^{n+1}$)
coincides with the one defined in~\eqref{deff-capp}.)

Next, the upper bounds in \eqref{eq-heat-cyl-n>=3}
and \eqref{eq-heat-cyl-n=2} follow
easily from the subadditivity of $\capp$
(\cite[Proposizione~1.1(ii)]{Lanc73} or \cite[Theorem~7.44]{watson}) as follows,
\[
\capp (C_{r,h}) \simle \frac{h}{r^2} \capp (C_{r,r^2})
\simeq \frac{h}{r^2} r^n = h r^{n-2}
\quad \text{when } h \ge r^2.
\]

For the lower bounds, 
assume that $h\ge r^2$ and
$n\ge2$.
Let $m_c$ be the $n$-dimensional Lebesgue measure restricted to $K_c:=B(0,r)\times \{c\}$.
By \cite[Proposizione~5.1]{Lanc73} (or \cite[Theorem~9]{watson78} or
\cite[Theorem~7.55]{watson} with $(a,b)=\R$), 
$m_c$ is the thermal capacitary distribution for $K_c$ and thus
\[
P^{m_c}(x_0,t_0) \le 1 \quad \text{for all } (x_0,t_0)\in \R^{n+1}\times (c,\infty)
\]
and $P^{m_c} (x_0,t_0)= 0$ for $t_0\le c$.
Moreover, for $t_0\in [c+kr^2,c+(k+1)r^2]$ with $k=1,2,\dots$ and all $x_0\in \R^n$,
\begin{equation*} 
P^{m_c}(x_0,t_0) \le P^{m_c}(0,t_0) 
           \le (4\pi kr^2)^{-n/2} \int_{B(0,r)} e^{-|x|^2/4(k+1)r^2}\, dx.
\end{equation*} 
The last integral is easily estimated using polar coordinates
and the elementary inequality $e^x \ge 1+x$  as
\begin{align*}
\int_{B(0,r)} e^{-|x|^2/4(k+1)r^2}\, dx
&\simle r^{n-2} \int_0^r e^{-\rho^2/4(k+1)r^2} \, \rho \,d\rho \\
&= 2(k+1) r^{n} (1-e^{-1/(k+1)}) \simle r^n.
\end{align*}
It thus follows that for the measure 
\[
\mu := \sum_{\substack{0\le jr^2\le h \\ j \in \Z}} m_{jr^2},
\]
obtained by placing copies 
of the $n$-dimensional Lebesgue measure 
at levels $jr^2$ in $C_{r,h}$,
the left-hand side  of~\eqref{eq-par-pot} can be estimated as
\begin{equation*} 
P^\mu(x_0,t_0)
\simle 1 + \sum_{j=0}^{k-1} 
\frac{r^n}{((k-j)r^2)^{n/2}}  
=  1 + \sum_{j=1}^k \frac{1}{j^{n/2}},
\end{equation*} 
where $k=\lfloor h/r^2 \rfloor$ is the largest integer $\le h/r^2$.

For $n\ge3$, the sum has an upper bound 
independent of $k$ and hence a multiple of $\mu$ is admissible 
in the definition of $\capp(C_{r,h})$.
It follows that 
\[
\capp(C_{r,h}) \simge \mu(C_{r,h}) \simeq kr^n \simeq h r^{n-2},
\]
which proves \eqref{eq-heat-cyl-n>=3}. 

Finally, for $n=2$ we instead get 
\[
P^\mu(x_0,t_0) 
\simle  1 + \log k \le 1+ \log(h/r^2)
\]
and consequently,
\[
\capp(C_{r,h}) \simge \frac{\mu(C_{r,h})}{1 + \log(h/r^2)} 
\simeq \frac{h}{1+\log(h/r^2)},
\]
which proves \eqref{eq-heat-cyl-n=2}.
\end{proof}

We are now ready to prove Proposition~\ref{prop-irreg-EG}.
Part~\ref{reg-n>=3} (i.e.\ $n \ge3$) will follow from the 
following more general result.

\begin{prop} \label{prop-Phi-rho}
Assume that $n\ge 3$ 
and let $\rho:(0,1]\to (0,\infty)$ be a nondecreasing function such that
$\lim_{\tau\to 0+} \rho(\tau) = 0$.
Then the origin is regular with respect to 
\begin{equation*} 
\Phi_{\rho} := \{(x,t) \in \R^{n} \times (-1,0): 
(-t)^{1/2} \rho(-t) <|x| < \rho(1)\}
\end{equation*}
if and only if 
\begin{equation}    \label{eq-div-int-rho}
\int_0^1 \rho(\tau)^{n-2} \, \frac{d\tau}{\tau} = \infty.
\end{equation}
\end{prop}

We first show 
how this leads to a proof of Proposition~\ref{prop-irreg-EG}\ref{reg-n>=3}.
A similar argument shows that the origin is regular with respect to the domain
\[
\biggl\{(x,t) \in \R^{n} \times (-\tfrac12,0): 
    0<\frac{-t}{| {\log(-t)} |^{\ga}} < \theta |x|^2 <\theta \biggr\}
    \quad \text{(with $\theta,\ga>0$)}
\]
for $n\ge3$ if and only if $\ga > 2/(n-2)$.

\begin{proof}[Proof of Proposition~\ref{prop-irreg-EG}\ref{reg-n>=3}]
This follows directly from Proposition~\ref{prop-Phi-rho} 
together with the observation that 
\[
\Theta_{l,\theta} = \Phi_{\rho} \quad \text{with} \quad
\rho(\tau) = \frac{\tau^{1/l-1/2}}{\theta^{1/l}}
\]
and thus the integral in~\eqref{eq-div-int-rho} converges if and only if $l<2$.
The regularity for $l \ge 2$ can also be obtained  from the tusk condition 
of Effros--Kazdan~\cite{EfKaz}.
\end{proof}

\begin{proof}[Proof of Proposition~\ref{prop-Phi-rho}]
Note that $(x,t)\in A_k$ is equivalent to 
\[
|x|^2 \le 4t \log \bigl(2^k(-4\pi t)^{n/2} \bigr),
\]
which requires $-t \le t_k:= 2^{-2k/n}/4\pi$.

To prove regularity, we first show,
for sufficiently large $k$, the inclusion 
\begin{equation} \label{eq-A_k-inc-1}
\itoverline{A_{k-1} \setm A_{k}}\setm \Theta_{l,\theta}
\supset
\itoverline{B(0,r_k)} \times [-(1+c)t_k,-t_k] =: C_k,
\quad \text{where } 
r_k= t_k^{1/2} \rho(t_k)
\end{equation}
and $c>0$ is some constant.
Indeed, 
\begin{align*}
F(r_k, (1+c)t_k) &= 
\biggl( \frac{1}{4\pi (1+c) t_k} \biggr)^{n/2}
\exp\biggl( -\frac{ t_k \rho(t_k)^2}{4(1+c)t_k} \biggr)  \\
&= \frac{2^k}{(1+c)^{n/2}} 
\exp \biggl( -\frac{\rho(t_k)^2}{4(1+c)}\biggr)
\ge 2^{k-1},
\end{align*}
whenever $(1+c)^{n/2}\le \tfrac32$ and the exponential is $\ge \tfrac34$, which 
holds for sufficiently large $k$ 
because $\lim_{\tau\to 0+} \rho(\tau) = 0$.
Note also  that  $ct_k \ge r_k^2$ for sufficiently large $k$ (again since 
$\lim_{\tau\to 0+} \rho(\tau) = 0$).
Moreover, we have the simple inclusion
\begin{equation} \label{eq-A_k-inc}
\itoverline{A_{k} \setm A_{k+1}}\setm \Phi_{\rho}
\subset \itoverline{B(0,r_k)} \times [-t_k,0] =: C'_k.
\end{equation}
Using 
\eqref{eq-heat-cyl-n>=3},
we see that, for $k \ge k_0$ with some sufficiently large $k_0$,
\[
\capp (C_k)  \simeq \capp (C'_k)  \simeq
t_k r_k^{n-2}
= t_k^{n/2} \rho(t_k)^{n-2}
\simeq 2^{-k} \rho(t_k)^{n-2}.
\]
Inserting this into the Wiener integral \eqref{eq-Wiener-EG},
we get
\[
\sum_{k=1}^\infty 2^k \capp \bigl( \itoverline{A_k \setm A_{k+1}}\setm \Phi_{\rho} \bigr) 
\simeq 
\sum_{k=1}^{k_0-1} 2^k \capp \bigl( \itoverline{A_k \setm A_{k+1}}\setm \Phi_{\rho} \bigr) 
+ \sum_{k=k_0}^\infty \rho(t_k)^{n-2}.
\]
Since the last series converges if and only if the integral in \eqref{eq-div-int-rho} does,
this concludes the proof.
\end{proof}

\begin{proof}[Proof of Proposition~\ref{prop-irreg-EG}\ref{reg-n=2}]
To prove regularity, we can assume that $l<2$,
by Theorem~\ref{thm-soda-reg} (or the comments after Theorem~\ref{thm:barrier-char}).
As in the proof of Proposition~\ref{prop-Phi-rho}
(with $n=2$ and $\rho(\tau)= \tau^{1/l-1/2}/\theta^{1/l}$ 
in~\eqref{eq-A_k-inc-1}), it can be seen
that there is $c>0$ such that for all sufficiently large $k$,
\begin{equation} \label{eq-A_k-inc-n=2}
\itoverline{A_{k-1} \setm A_{k}}\setm \Phi_{\rho}
\supset
\itoverline{B(0,r_k)} \times [-(1+c)t_k,-t_k] =: C_k,
\quad \text{where } 
r_k= \biggl(\frac{t_k}{\theta}\biggr)^{1/l}
\end{equation}
and $t_k:= 2^{-k}/4\pi$.
Note that $ct_k \ge r_k^2$ for sufficiently large $k$ (since $l<2$).

Thus, \eqref{eq-heat-cyl-n=2} implies that, for 
$k \ge k_0$ with some sufficiently large $k_0$,
\[
\capp (C_k) \simge \frac{ ct_k}{1+ \log (ct_k/r_k^2)}
=  \frac{ ct_k}{1+ \log (c  \theta^{2/l} t_k^{1-2/l})}
\simeq
\frac{2^{-k}}{1+\log ( 2^{(2/l-1)k} )}
\simeq \frac{2^{-k}}{k}.
\]
Inserting this and \eqref{eq-A_k-inc-n=2} into the Wiener integral \eqref{eq-Wiener-EG} 
we get
\[
\sum_{k=1}^\infty 2^k \capp \bigl( \itoverline{A_{k-1} \setm A_{k}}\setm \Theta_{l,\theta} \bigr) 
\simge \sum_{k=k_0}^\infty \frac{1}{k}
= \infty.
\]
Thus the origin is regular for every $\Theta_{l,\theta}$, by the Wiener criterion.

Finally, the irregularity with respect to the punctured cylinder
\[
\Theta_0:=(B(0,1)\setm\{0\})\times(-1,0), 
\]
follows directly from 
the result by Babu\v{s}ka--V\'yborn\'y~\cite[S\"atze~3 and~4]{BabVyb} 
saying that a point on the lateral boundary of a cylinder is regular if and only if
the corresponding base point is regular for harmonic functions.
However, in this simple case we can give a more elementary direct proof.
Consider
the continuous boundary data $f(x,t)=1+t-|x|$ on $\bdy\Theta_0$.
Recall that $u(x)=\log(1/|x|)$ is harmonic in $\R^2 \setm \{0\}$.
It then follows that for $\eps>0$,
\[
      v_\eps(x,t)=\begin{cases}
          \min\{\eps u(x),1\}, & \text{if } -1 < t \le -\eps, \\
          1, & \text{if } -\eps < t < 0, 
          \end{cases}
\]
is a $2$-superparabolic function in $\Theta_0$.
It is easy to see that $v_\eps$ is  competing in the 
definition of $\uP f$ and thus $\uP f \le v_\eps$.
Letting $\eps \to 0$ shows that $\uP f \le 0 <1 =f(0,0)$ in $\Theta_0$ and thus 
the origin is irregular.
\end{proof}

The inclusions \eqref{eq-A_k-inc-1} and \eqref{eq-A_k-inc} show
that in order to get irregularity results for $n=2$ and 
domains punctured by very sharp cusps,
one needs better estimates than \eqref{eq-heat-cyl-n=2} for ``tall'' cylinders.
In particular, the simple upper bound in \eqref{eq-heat-cyl-n=2} gives
a diverging Wiener type integral and cannot be used to prove irregularity
for very sharp cusps nor for the punctured cylinder in
Proposition~\ref{prop-irreg-EG}\ref{reg-n=2}.

\section{Power type barriers}
\label{sect-trad-bar}

The following lemma gives a 
barrier $v_\ka$ as in~\eqref{eq-trad-barrier} 
for a class of soda cans not covered by
Theorem~\ref{thm-soda-reg}.
However, because of nonlinearity, multiples $Mv_\ka$ 
of this barrier are not \p-superparabolic when $M>1$,
so it is far from clear whether there exists a barrier family as in Definition~\ref{def-barrier}.

\begin{lem}  \label{lem-simple-barrier}
Assume that $p>2$ and $\theta>0$.
If $l\ge p/(p-1)$, then
for every 
\begin{equation*}  
0<\ka< \biggl( \frac{p-1}{np\theta}  \biggr)^{1/(p-2)} =: \kat,
\end{equation*}
the function
\begin{equation} \label{eq-vka}
v_\ka (x,t) := \frac{\ka(p-1)}{p}  |x|^{p/(p-1)} + n \ka^{p-1}t 
\end{equation}
is a \p-parabolic barrier for $\Theta_{l,\theta}$ at $(0,0)$.
\end{lem}

\begin{proof}
Lemma~\ref{lem-Delta-powers} and a direct calculation show 
that $v_\ka$ is \p-parabolic in $\Theta_{l,\theta}$.
For $t\ge -\theta|x|^l \ge -\theta$ and $\ka<\kat$, 
we have that
\begin{align*}
v_\ka(x,t) 
&\ge  \ka |x|^{p/(p-1)} \biggl( \frac{p-1}{p} - n\theta \ka^{p-2} |x|^{l-\frac{p}{p-1}}  \biggr) \\
&\ge  \ka n\theta (\kat^{p-2} - \ka^{p-2}) |x|^{p/(p-1)},
\end{align*}
since $l-p/(p-1)\ge0$.
As $v_\ka$ is continuous at $(0,0)$ and $v_\ka(0,0)=0$, it becomes a  barrier.
\end{proof}

In \cite[Proposition~4.2]{BBGP}, the functions $v_\ka$ 
were used with large $\ka$ to provide barrier families (and thus regularity) in the following 
two situations: For $l\ge p/(p-1)>
2$ (and thus $p<2$) and for $l>p>2$.
By the very definition of Perron solutions, the barriers in 
Lemma~\ref{lem-simple-barrier} imply regularity of $(0,0)$ for 
boundary data satisfying {$|f-f(0,0)| \le v_\ka$} on $\bdy\Theta_{l,\theta}$.
By using the pasting Lemma~\ref{lem-pasting}, we can use these functions
more efficiently to obtain partial boundary regularity for $(0,0)$ as 
in Proposition~\ref{prop-small-reg}.

\begin{proof}[First proof of Proposition~\ref{prop-small-reg},
based on power type barriers]
Lemma~\ref{lem-Delta-powers} and a direct calculation show 
that for every $\ka>0$, the function $v_\ka$, defined in \eqref{eq-vka},
is \p-parabolic in $\Theta_{l,\theta}$. 
For $t\ge -\theta|x|^l $ and $|x|\le \de \le 1$,  we have that
\[
v_\ka(x,t) 
\ge \biggl( \frac{\ka(p-1)}{p}  - n\theta \ka^{p-1} \de^{l-\frac{p}{p-1}}  \biggr) |x|^{p/(p-1)}
\]
with the coefficient on the right-hand side being maximal when
\[
\ka = \biggl( \frac{\de^{\frac{p}{p-1}-l}}{np\theta} \biggr)^{1/(p-2)}.
\]
For this choice of $\ka$,
\begin{align*}
v_\ka(x,t) &\ge \biggl( \frac{\de^{\frac{p}{p-1}-l}}{np\theta} \biggr)^{1/(p-2)}
\biggl( \frac{p-1}{p}  - n\theta \frac{\de^{\frac{p}{p-1}-l}}{np\theta} \de^{l-\frac{p}{p-1}}  \biggr) |x|^{p/(p-1)}  \\
  &=   \frac{p-2}{p} \biggl( \frac{\de^{\frac{p}{p-1}-l}}{np\theta} \biggr)^{1/(p-2)} |x|^{p/(p-1)}. 
 \end{align*}
For $|x|=\de$ we then have
\begin{equation}
v_\ka(x,t) 
\ge \frac{p-2}{p} \biggl( \frac{\de^{\frac{p}{p-1}-l}}{np\theta} \biggr)^{\frac{1}{p-2}} \de^{\frac{p}{p-1}}
= \frac{p-2}{p} \biggl( \frac{1}{np\theta} \biggr)^{\frac{1}{p-2}} \de^{\frac{p-l}{p-2}} 
 =: m_\de >0.
\label{eq-def-m-de}
\end{equation}
The pasting Lemma~\ref{lem-pasting} 
implies that the function
\[
    \vt_\de(x,t):=\begin{cases}
   \min\{v_\ka(x,t),m_\de\}, & \text{if }  |x| <  \de, \\
   m_\de, & \text{if }  |x| \ge  \de, 
     \end{cases}
\]
is \p-superparabolic in the set   
\[
\{(x,t): 0< -t < \theta|x|^l \} \supset \Theta_{l,\theta}.
\]
Hence $\vt_\de$ is a barrier for $\Theta_{l,\theta}$ at $(0,0)$.

It follows that the functions $f(0,0) \pm \vt_\de$ are admissible 
in the definitions of $\uP f$ and $\lP f$,
respectively, for every $f$ satisfying \eqref{eq-bound-f-ka}.
From this, together with
\eqref{eq-lp<=uP}
and the continuity of $\vt_\de$, \eqref{eq-lim-Hf} follows immediately.
\end{proof}

\begin{proof}[Proof of Corollary~\ref{cor-reg-l=p}]
We may assume that $f(0,0)=0$.
Let $\eps>0$.
Then there is $0<\de<1$ such that $|f(x,t)|<\eps$ for $(x,t) \in \bdy \Theta_{l,\theta}$
with $|x|<\de$.
Let $f_\eps:=(f-\eps)_+$.
Since $l=p$ and thus
\[
   \frac{\frac{p}{p-1}-l}{p-2}+\frac{p}{p-1}=0,
\]
we get from \eqref{eq-ass-with-Mtheta} that 
\begin{equation*}  
f_\eps \le \frac{p-2}{p} \biggl( \frac{\de^{\frac{p}{p-1}-l}}{np\theta} \biggr)^{1/(p-2)} 
\min\{|x|,\de\}^{p/(p-1)} 
\quad \text{on } \bdy \Theta_{l,\theta},
\end{equation*}
i.e.\ the assumption~\eqref{eq-bound-f-ka} holds.
Thus it follows from Proposition~\ref{prop-small-reg} that
\begin{equation*} 
\limsup_{\Theta_{l,\theta} \ni \xi \to (0,0)} \uP_{\Theta_{l,\theta}} f(\xi) 
\le \eps + \limsup_{\Theta_{l,\theta} \ni \xi \to (0,0)} \uP_{\Theta_{l,\theta}} f_\eps(\xi) 
= \eps.
\end{equation*}
Letting $\eps \to 0$, applying this also to $-f$
and using \eqref{eq-lp<=uP}
shows that \eqref{eq-lim-Hf} holds.
\end{proof}

\section{Barriers from the Barenblatt solution}
\label{sect-Baren}

The so-called \emph{Barenblatt solution} \cite{Bare52} 
of the \p-parabolic equation \eqref{eq:para}, with $p>2$,
is for $t>0$ and $\la:=n(p-2)+p$,
defined by
\begin{align*}
B_p(x,t) = B_{p,C}(x,t) & =t^{-n/\la}\biggl[C-
\frac{p-2}{p\la^{1/(p-1)}}
\biggl(\frac{|x|}{t^{1/\la}}\biggr)^{p/(p-1)}\biggr]_{+}^{(p-1)/(p-2)} \\
& = \biggl[ C t^{\frac{n(2-p)}{\la (p-1)}} - \frac{p-2}{p\la^{1/(p-1)}}
\frac{|x|^{p/(p-1)}}{t^{1/(p-1)}}\biggr]_{+}^{(p-1)/(p-2)},
\end{align*}
where $C>0$ is arbitrary.
The Barenblatt solution $B_p$ is \p-parabolic in the
whole upper half-space $\Rn\times (0,\infty)$.

We will now use the Barenblatt solution 
(with $C=1$) to construct
a suitable barrier $\wt_\eps$ for $\Theta_{\theta,l}$
in the range $p/(p-1) \le l \le p$. 
To be able to work with the Barenblatt solution in its usual coordinate form,
$\wt_\eps$ will be a barrier at $(0,t_\eps)$ for the translated set
\[
\{(x,t): t_\eps-\theta |x|^l < t<t_\eps \text{ and } |x|<1\}.  
\]

Note that 
\[
B_{p,C} (x,t) = B_{p,1} (x/A,t/A^p),
\quad \text{where} \quad
A= C^{\frac{\la(p-1)}{np(p-2)}}  = C^{\frac{\la(p-1)}{p(\la-p)}}.
\]
Thus, in the arguments below, we could have used any $C>0$ in the definition
of the Barenblatt solutions.
Using $C=1$ makes some formulas simpler, but the conclusions remain
the same.

\begin{lem} \label{lem-Bar-w}
Assume that $p>2$ and $p/(p-1)\le l <\la$.
Let $0<\eps<1$, 
\begin{equation}      \label{eq-theta-M}
 \theta  
 =\frac{\la^{(p-2)/(p-1)}}{np}
 \quad \text{and} \quad
   M = \frac{p-2}{p \la^{1/(p-1)}}.
\end{equation}
Then, for sufficiently small $\de_\eps>0$ and for 
\begin{equation}  \label{eq-t_eps}
t_\eps \ge \de_\eps^{\frac{\la (l(p-1)-p)}{(p-2)(\la+n)}},
\end{equation}
the function
\[
w_\eps(x,t):= t_\eps^{-n/\la} - B_p(x,t),
\]
with $C=1$ in the definition of $B_p$,
is  a \p-parabolic barrier at $(0,t_\eps)$ for the set
\[
\Theta_\eps:=\{(x,t): 
t_\eps-\theta |x|^l< t < t_\eps
\text{ and }  |x| <  \de_\eps\}.  
\]
Moreover, for $(x,t)\in \overline{\Theta}_\eps$,
\begin{equation}   \label{eq-weps-new}
w_{\eps}(x,t) \ge \frac{M(1-\eps)^2 |x|^{p/(p-1)}}{ t_\eps^{\frac{\la+n}{\la(p-1)}} }.
\end{equation}
In particular, for $|x|=\de_\eps$ and 
$t_\eps= \de_\eps^{\frac{\la (l(p-1)-p)}{(p-2)(\la+n)}}$,
\begin{equation}   \label{eq-est-weps-at-de-eps-teps}
w_{\eps}(x,t_\eps) \ge M(1-\eps)^2 \de_\eps^{(p-l)/(p-2)}.
\end{equation}
\end{lem}

\begin{proof}
Let $0 <\eps<1$ and 
\begin{equation}   \label{eq-alpha}
\al := 1-\frac{p}{\la(p-1)} = \frac{(p-2)(\la+n)}{\la (p-1)} >0.
\end{equation}
Then \eqref{eq-t_eps} implies that
\begin{equation}    \label{eq-t-eps-de-eps}
\de_\eps^{l-\frac{p}{p-1}}
\le t_\eps^\al =  t_\eps^{\frac{(p-2)(\la+n)}{\la(p-1)}},
\end{equation}
and therefore (using also~\eqref{eq-alpha})
\begin{equation}   \label{eq-choose-small-|x|}
\frac{|x|^l}{t_\eps} 
\le \biggl( \frac{|x|}{t_\eps^{1/\la}} \biggr)^{p/(p-1)}
\quad \text{if } |x| \le \de_\eps.
\end{equation}

Let $\xi=(x,t)\in \overline{\Theta}_\eps$, i.e.\ $|x| \le \de_\eps$
and $t_\eps-\theta |x|^l\le t \le t_\eps$. 
Then, with $M$ as in~\eqref{eq-theta-M},
\begin{align}
B_p(\xi) & =  \biggl[ t^{\frac{n(2-p)}{\la (p-1)}} - M
\frac{|x|^{p/(p-1)}}{t^{1/(p-1)}}\biggr]_{+}^{(p-1)/(p-2)}  \nonumber\\
& \le  \biggl[ (t_\eps - \theta|x|^l)^{\frac{n(2-p)}{\la (p-1)}} - 
M \frac{|x|^{p/(p-1)}}{t_\eps^{1/(p-1)}}   \biggr]_{+}^{(p-1)/(p-2)} \nonumber  \\
&= t_\eps^{-n/\la} \biggl[ \biggl(1-\frac{\theta |x|^l}{t_\eps}\biggr)^{\frac{n(2-p)}{\la (p-1)}} 
  - M \biggl( \frac{|x|}{t_\eps^{1/\la}} \biggr)^{p/(p-1)}\biggr]_{+}^{(p-1)/(p-2)}.
\label{eq-Bp1}
\end{align}
Choose $0<\de_0<1$ so that for all $0\le \de \le \de_0$ both
\begin{equation} \label{eq-1}
(1-\de)^{\frac{n(2-p)}{\la (p-1)}} \le 1+(1+\eps(p-2))
\frac{n(p-2)}{\la (p-1)} \de
\end{equation}
and 
\begin{equation} \label{eq-2} 
(1-\de)^{(p-1)/(p-2)} \le 1-\frac{(1-\eps)(p-1)}{p-2} \de.
\end{equation}

Note that, with our choice of $\theta$ and $M$  in~\eqref{eq-theta-M},
\begin{equation} \label{eq-3}
\frac{n(p-2)\theta}{\la } = \frac{p-2}{p \la^{1/(p-1)} } = M.
\end{equation}
Hence, if both
\begin{equation}    \label{eq-assume-small}
\frac{\theta |x|^l}{t_\eps} \le \de_0
\quad \text{and} \quad
\frac{M(1-\eps)(p-2)}{p-1}  \biggl( \frac{|x|}{t_\eps^{1/\la}} \biggr)^{p/(p-1)} \le \de_0,
\end{equation}
then we get by 
\eqref{eq-choose-small-|x|}--\eqref{eq-3}
that (still assuming that $\xi=(x,t)\in \overline{\Theta}_\eps$)
\begin{align}
B_p(\xi) 
&\os{\eqref{eq-Bp1},\eqref{eq-1}}{\le}  t_\eps^{-n/\la}  \biggl[ 1 + (1+\eps(p-2))
     \frac {n(p-2)}{\la (p-1)} \frac{\theta|x|^l}{t_\eps} - M
\biggl( \frac{|x|}{t_\eps^{1/\la}} \biggr)^{p/(p-1)} \biggr]_{+}^{(p-1)/(p-2)} \nonumber \\
&\os{\eqref{eq-3}}{=}  t_\eps^{-n/\la}  \biggl[ 1 +        \frac{(1+\eps(p-2))M}{p-1} \frac{|x|^l}{t_\eps} - M
               \biggl( \frac{|x|}{t_\eps^{1/\la}} \biggr)^{p/(p-1)} \biggr]_{+}^{(p-1)/(p-2)} \nonumber \\
&\os{\eqref{eq-choose-small-|x|}}{\le} t_\eps^{-n/\la}  \biggl[ 1 - \frac{M(1-\eps)(p-2)}{p-1}
\biggl( \frac{|x|}{t_\eps^{1/\la}} \biggr)^{p/(p-1)} \biggr]_{+}^{(p-1)/(p-2)} \nonumber \\  
&\os{\eqref{eq-2}}{\le} t_\eps^{-n/\la}  \biggl[ 1 - M(1-\eps)^2
\biggl( \frac{|x|}{t_\eps^{1/\la}} \biggr)^{p/(p-1)} \biggr]_{+}. 
\label{eq-Bp<1-new}
\end{align}
Moreover, for sufficiently small $\de_\eps$ (and $|x|\le \de_\eps$),   
we have from \eqref{eq-t-eps-de-eps} that
\begin{equation} \label{eq-Bp>0-inclusion-newer}
\frac{|x|}{t_\eps^{1/\la}} \le
\frac{\de_\eps}{ \bigl(\de_\eps^{1/\al\la} \bigr)^{l-\frac{p}{p-1}} }  
=: \de_\eps^\ga
\le \biggl(  \frac{(p-1)\de_0}{M(1-\eps)(p-2)}  \biggr)^{1-1/p},
\end{equation}
because by \eqref{eq-alpha} and the assumption $l < \la$,
\[
\ga := 1-  \frac{l-p/(p-1)}{\al \la} 
= \frac{\la - p/(p-1) - l +p/(p-1)}{\al \la}
= \frac{\la - l}{\al \la}
> 0.
\]
Since also \eqref{eq-choose-small-|x|} holds, we conclude that
\eqref{eq-assume-small} and thus
\eqref{eq-Bp<1-new} hold, provided that 
$\de_\eps $ is sufficiently small.

If we also choose $\de_0$ small  enough, then  by
\eqref{eq-Bp>0-inclusion-newer},
\begin{equation*} 
\biggl(\frac{|x|}{t_\eps^{1/\la}}\biggr)^{p/(p-1)}
\le  \frac{1}{M(1-\eps)^2} 
\end{equation*}
and thus \eqref{eq-Bp<1-new} holds without taking
the positive part in the last line, i.e.
\begin{equation}
B_p(\xi) 
\le t_\eps^{-n/\la}  \biggl( 1 - M(1-\eps)^2
\biggl( \frac{|x|}{t_\eps^{1/\la}} \biggr)^{p/(p-1)} \biggr). 
\label{eq-Bp<1-newer}
\end{equation}

Next,  the function
\[
 w_\eps(x,t) = t_\eps^{-n/\la} - B_p(x,t)=
B_p(0,t_\eps) - B_p(x,t)
\]
is \p-parabolic in $\Theta_\eps$
and thus \p-superparabolic therein.
By the continuity of $B_p$, we have
\begin{equation} \label{eq-wt_eps-lim-new}
\lim_{\Theta_\eps \ni \zeta \to (0,t_\eps)} w_\eps(\zeta) = w_\eps(0,t_\eps) = 0,
\end{equation}
so the first requirement in the definition \eqref{eq-trad-barrier} of 
a barrier holds.
On the other hand, from \eqref{eq-Bp<1-newer}
we conclude that 
for all $\xi =(x,t) \in \overline{\Theta}_\eps$ with  $x\ne0$,
\begin{align} 
\liminf_{\Theta_\eps \ni \zeta \to \xi} w_{\eps}(\zeta) 
&= w_{\eps}(\xi) 
= t_\eps^{-n/\la} - B_p(\xi) \nonumber\\
& \ge t_\eps^{-n/\la} \biggl( 1- \biggl(1- M(1-\eps)^2 
\biggl( \frac{|x|}{t_\eps^{1/\la}} \biggr)^{p/{p-1)}}
\biggr)
\biggr) \nonumber \\
&= M(1-\eps)^2\frac{|x|^{p/(p-1)}}{ t_\eps^{(\la+n)/\la(p-1)} } 
> 0.  \label{eq-weps-newer}
\end{align}
This proves \eqref{eq-weps-new} and
concludes the construction of the barrier $w_\eps$.

Finally, for $|x|=\de_\eps$ and $t_\eps= \de_\eps^{(l-p/(p-1))/\al}$,
it follows from 
\eqref{eq-weps-newer} that
\[
w_{\eps}(x,t_\eps) \ge
\frac{ M(1-\eps)^2 \de_\eps^{p/(p-1)}}{ \bigl(\de_\eps^{(l-p/(p-1))/\al} \bigr)^{(\la+n)/\la(p-1)} }   
= M(1-\eps)^2 \de_\eps^\tau,
\]
where, by \eqref{eq-alpha},
\begin{equation*}   
\tau := \frac{p}{p-1} - \frac{(l-p/(p-1))(\la+n)}{\al \la (p-1)} 
= \frac{p}{p-1} - \frac{l-p/(p-1)}{p-2} 
= \frac{p-l}{p-2}.\qedhere
\end{equation*}
\end{proof}

The following corollary and its proof show how Lemma~\ref{lem-Bar-w} can be used
to obtain a barrier family for $\Theta_{l,\theta}$ when $l>p>2$.
This gives a different proof for a part of Theorem~\ref{thm-soda-reg}.

\begin{cor}  \label{cor-barrier-family-Barenblatt}
If $l>p>2$ then $(0,0)$ is regular with respect to $\Theta_{l,\theta}$
for all $\theta>0$.
\end{cor}

\begin{proof}
It suffices to consider $l<\la$,
since regularity with respect to a larger domain implies regularity 
with respect to a smaller one,
see the comments after Theorem~\ref{thm:barrier-char}.
Let $\theta$, $M$, $w_\eps$ and $\de_\eps$ be as in Lemma~\ref{lem-Bar-w}
and let $\tau=(p-l)/(p-2)$.
The  pasting Lemma~\ref{lem-pasting}, together with the 
estimate~\eqref{eq-est-weps-at-de-eps-teps}, then implies  that the function,
defined for $(x,t)\in\Theta_{l,\theta}$ by
\begin{equation}   \label{eq-def-wt}
    \wt_\eps(x,t):=\begin{cases}
   \min\{w_{\eps}(x,t+t_\eps),M(1-\eps)^2\de_\eps^\tau\}, & \text{if }  
    |x| <  \de_\eps, \\
   M(1- \eps)^2\de_\eps^\tau, & \text{if }  
   |x| \ge  \de_\eps,
    \end{cases}
\end{equation}
is \p-superparabolic in $\Theta_{l,\theta}$.
We will now fix $0<\eps<1$ in Lemma~\ref{lem-Bar-w},
but will vary $\de_\eps$.
Since $\tau<0$ in \eqref{eq-def-wt} and $l(p-1)-p>0$, the choice
\[
\de_\eps= \de_{\eps,j}
= 1/j \quad \text{and}  \quad 
t_\eps= t_{\eps,j} 
= (1/j)^{\frac{\la (l(p-1)-p)}{(p-2)(\la+n)}}
\]
in Lemma~\ref{lem-Bar-w} 
gives, for sufficiently large $j$, a barrier family at $(0,0)$
for $\Theta_{l,\theta}$ with this particular choice of $\theta$.
Lemma~\ref{lem-scaling-gen} then guarantees a barrier family 
for all $\theta>0$.
Hence $(0,0)$ is regular with respect to $\Theta_{l,\theta}$ for all $\theta >0$,
by Theorem~\ref{thm:barrier-char}.
\end{proof}

We are now ready to give our second proof of 
Proposition~\ref{prop-small-reg}.

\begin{proof}[Second proof of Proposition~\ref{prop-small-reg},
based on Barenblatt type barriers]
If $l >p>2$, then the full regularity follows
directly from Corollary~\ref{cor-barrier-family-Barenblatt},
whose proof also provides a Barenblatt type barrier family.
Assume therefore that $p/(p-1)\le l\le p$.

Let 
$\theta>0$ and $0<\eps < \de \le1$  be arbitrary and set
\[
\theta_0 =\frac{\la^{(p-2)/(p-1)}}{np}.
\]
Let $\de_\eps>0$, $t_\eps= \de_\eps^{\frac{\la (l(p-1)-p)}{(p-2)(\la+n)}}$, 
\[
   M = \frac{(p-2)}{p \la^{1/(p-1)}}
  \quad \text{and} \quad 
w_\eps(x,t):= t_\eps^{-n/\la} - B_p(x,t)
\]
be as in Lemma~\ref{lem-Bar-w}.
Then $w_\eps$ is \p-parabolic in the set
 \[
\{(x,t): t_\eps - \theta_0 |x|^l< t < t_\eps
\text{ and }  |x| <  \de_\eps\}, 
\]
by Lemma~\ref{lem-Bar-w}.
For
\[
a=\frac{\de_\eps}{\de}, \quad 
b= \frac{\theta_0}{\theta} a^l  
\quad \text{and} \quad
A = \biggl( \frac{b}{a^p} \biggr)^{1/(p-2)} 
= \biggl( \frac{\theta_0}{\theta} \biggl( \frac{\de_\eps}{\de} \biggr)^{l-p} \biggr)^{1/(p-2)},
\]
let 
\begin{equation}   \label{eq-def-u-eps-from-w}
u_\eps(x,t) = A w_\eps(ax, bt+t_\eps).
\end{equation}
Then by the scaling Lemma~\ref{lem-scaling-gen},
$u_\de$ is \p-parabolic in 
the restricted soda can domain $\{(x,t) \in \Theta_{l,\theta}: |x|<\de\}$.
Moreover, by \eqref{eq-theta-M}, \eqref{eq-weps-new} and noting that
\[
\frac{l-p}{p-2}+\frac{p}{p-1} = \frac{l-p/(p-1)}{p-2},
\]
we have
for $(x,t)\in \overline{\Theta}_{l,\theta}$ with $|x|\le\de$,
\begin{align}
u_\eps(x,t)  
&\ge \biggl( \frac{\theta_0}{\theta} \biggl( \frac{\de_\eps}{\de} \biggr)^{l-p} \biggr)^{1/(p-2)} 
   \frac{M(1-\eps)^2 (a|x|)^{p/(p-1)}} { \Bigl( \de_\eps^{\frac{\la (l(p-1)-p)}{(p-2)(\la+n)}} 
   \Bigr)^{(\la+n)/\la(p-1)}} \nonumber \\
&= \frac{\displaystyle   \biggl( \frac{\la^{(p-2)/(p-1)}}{np\theta} \biggr)^{1/(p-2)}  
    \frac{(p-2)(1-\eps)^2}{p\la^{1/(p-1)}} 
\biggl( \frac{\de_\eps}{\de} \biggr)^{\frac{l-p}{p-2}+\frac{p}{p-1}} }
{\de_\eps^{\frac{l-p/(p-1)}{p-2}}}|x|^{p/(p-1)} \nonumber \\
&= \frac{(p-2)(1-\eps)^2}{p} \biggl( \frac{\de^{\frac{p}{p-1}-l}}{np\theta} \biggr)^{1/(p-2)}
|x|^{p/(p-1)},
\label{eq-est-u-eps}
\end{align}
which, up to the factor $(1-\eps)^2$,
gives the same coefficient in front of $|x|^{p/(p-1)}$ 
as in~\eqref{eq-bound-f-ka}.
In particular, with 
\begin{equation*}   
M_\theta=\frac {p-2}{p} \biggl(\frac{1}{np\theta}\biggr)^{1/(p-2)}
\end{equation*}
as in \eqref{eq-M-theta} and 
$\tau=(p-l)/(p-2)\ge0$, we get as in \eqref{eq-def-m-de} that
\begin{equation*}   
u_\eps(x,t) \ge  M_\theta(1-\eps)^2\de^\tau,
\quad \text{when } |x|=\de.
\end{equation*}
The  pasting Lemma~\ref{lem-pasting} then shows that the function
\begin{equation}   \label{eq-def-ut}
    \ut_\eps(x,t):=\begin{cases}
   \min\{u_{\eps}(x,t),M_\theta(1-\eps)^2\de^\tau\}, & \text{if }  
(x,t)\in\Theta_{l,\theta} \text{ and }  
|x| <  \de,
\\
   M_\theta(1- \eps)^2\de^\tau, & \text{if }  
(x,t)\in\Theta_{l,\theta} \text{ and }
|x| \ge  \de,
    \end{cases}
\end{equation}
is \p-superparabolic in $\Theta_{l,\theta}$ and is thus a barrier.

To conclude the proof of \eqref{eq-lim-Hf}, let $f$ be as in~\eqref{eq-bound-f-ka}.
We can assume that $f(0,0)=0$.
Then, \eqref{eq-est-u-eps} and \eqref{eq-def-ut} imply that
for sufficiently small $\eps>0$,
\begin{align*}
|f| &\le M_\theta \bigl( (1-\eps)^2 + (2\eps -\eps^2)\bigr) \de^{(\frac{p}{p-1}-l)/(p-2)}
    \min\{|x|,\de\}^{p/(p-1)}  \\
&\le \ut_\eps  + (2\eps - \eps^2) M_\theta \de^\tau 
\quad \text{on }  \bdy \Theta_{l,\theta},
\end{align*}
and hence, by the definition of Perron solutions,
\[
   \uP_{\Theta_{l,\theta}} f(x,t) \le \ut_\eps  + (2\eps - \eps^2) M_\theta \de^\tau 
\quad \text{in } \Theta_{l,\theta}.
\]
Using \eqref{eq-wt_eps-lim-new} with \eqref{eq-def-u-eps-from-w}
and then letting $\eps \to 0$ shows that
\[
     \limsup_{\Theta_{l,\theta} \ni \xi \to (0,0)} \uP_{\Theta_{l,\theta}} f(\xi)  \le 0.
\]
Applying this also to $-f$, and using \eqref{eq-lp<=uP}, yields
\begin{equation*}
0 \le \lim_{\Theta_{l,\theta} \ni \xi \to (0,0)} \lP_{\Theta_{l,\theta}} f(\xi)  
\le      \lim_{\Theta_{l,\theta} \ni \xi \to (0,0)} \uP_{\Theta_{l,\theta}} f(\xi) 
\le 0,
\end{equation*}
i.e.\ \eqref{eq-lim-Hf} holds.
\end{proof}

\section{Real-world interpretations}
\label{sect-real-world}

We end the paper with two naive situations
of how the Dirichlet problem in soda cans and the regularity of the origin 
can be interpreted in real-world problems.

\begin{example}  \label{ex-diffusion}
\emph{Nonlinear diffusion in a medium.}
Typically, $n=3$.
At each time $t\in(-t_0,0)$, 
let the ball $B^t=B(0,r_t)$ consist of a substance that 
is being released to (or dissolving in) the surrounding medium, where it diffuses 
according to the \p-parabolic equation.
The solution $u(x,t)$ of this equation in the complement of $\bigcup_{t\in(-t_0,0)}B^t$
describes the concentration  
of the substance in position $x\notin B^t$ at time $t\in(-t_0,0)$.

The initial concentration at $t=-t_0$ is $u(x,-t_0)=0$ for $x\notin B^{-t_0}$.
On the boundary $\bdy B^t$, the concentration is 
a known function $f$ (continuous in space and time),
e.g.\  a fixed constant.

As the substance releases into the medium, the ball $B^t$ containing it
shrinks to eventually become a point and disappear at the final time $t=0$.
In the case of soda can domains $\Theta_{l,\theta}$, studied in this paper, the speed of 
shrinking is given by 
\[
r_t=\biggl(\frac{-t}{\theta}\biggr)^{1/l},  \quad -t_0<t<0.
\]

Regularity of the final point $(0,0)$ with respect to the soda can $\Theta_{l,\theta}$
then means that just before all the substance 
has been released, its concentration near
$x=0$ is close to $f(0,0)$.
Irregularity means a possible jump or discontinuity of the concentration
for some continuous data $f$.

For example, the \p-superparabolic function $u$ in the proof of Proposition~\ref{prop-soda-irreg}
shows that immediately after  the last substance has been released,
its concentration is zero at every $x\ne0$.
Indeed, $p<2$ in this example means fast diffusion.

(The soda cans $\Theta_{l,\theta}$ also impose
boundary conditions at $|x|=1$, which we ignore here. 
Indeed, a simple comparison between Perron solutions shows that
regularity at $(0,0)$ holds 
equivalently also for discontinuous bounded boundary data
as long as they are continuous at $(0,0)$, 
cf.\  the proof of Lemma~9.6 in  Heinonen--Kilpel\"ainen--Martio~\cite{HeKiMa}.)
\end{example}

\begin{example}  \label{ex-shallow-layer}
\emph{Heat equation with  a partially immersed heater/cooler.}
Typically, $n=2=p$.
In this interpretation, a hot or cool body
passes through a shallow layer of a medium, represented by
the 2-dimensional time slices of $\Theta_{l,\theta}$, 
until it completely leaves the medium at time $t=0$.
For example, a fading stream of magma is passing through
an underground layer of water in a geological slab.
At each time $t\in (-t_0,0)$, the body intersects the shallow layer of the medium 
in an essentially $2$-dimensional ball $B^t=B(0,r_t)$.
The parameters $l$ and $\theta$  in $\Theta_{l,\theta}$ describe
how fast this cross-section changes in time.

The body has certain temperature (continuous in space and time), 
which provides the 
boundary data $u(x,t)$ on $\bdy B^t$ at time $t$.
The heat transfer in the medium exterior to the body follows the 
heat equation with $p=2$.

Proposition~\ref{prop-irreg-EG}\ref{reg-n=2} shows that for all 
$l,\theta>0$, the final temperature of the heater/cooler is  
reflected in the temperature of the medium 
close to the separation of the body from the medium,
but this need not be the case if the heater/cooler is  an infinitesimally
thin fibre.
\end{example}


\end{document}